\documentclass[12pt, a4paper]{amsart}
\usepackage[latin1]{inputenc}
\usepackage{a4wide}
\pdfoutput=1

\newif
\iffinal
\finaltrue

\iffinal
\newcommand{\mscexp}[1]{}
\else
\usepackage[eso-foot]{svninfo}
\svnInfo $Id: carries.tex 771 2016-02-24 15:27:46Z cheuberg $ 
\newcommand{\mscexp}[1]{ [\protect#1]}
\fi

\usepackage{appendix}
\usepackage{lineno}
\usepackage[T1]{fontenc}
\usepackage{amsthm}
\usepackage{amsfonts}
\usepackage{graphicx}
\usepackage{amssymb}
\usepackage{amscd}
\usepackage{amsmath}
\usepackage{algorithm}
\usepackage{algpseudocode}
\usepackage{hyperref}
\usepackage{pdflscape}
\usepackage{tikz}
\usepackage{pgfplots}
\usetikzlibrary{automata}
\usepackage{array}

\newtheorem{theorem}{Theorem}
\newtheorem{lemma}{Lemma}[section]
\theoremstyle{definition}
\newtheorem{definition}[lemma]{Definition}
\theoremstyle{remark}
\newtheorem{example}[lemma]{Example}
\newtheorem{remark}[lemma]{Remark}

\newcommand{\calA}{\mathcal{A}}
\newcommand{\bigOh}{\mathcal{O}}
\newcommand{\calL}{\mathcal{L}}
\newcommand{\calM}{\mathcal{M}}
\newcommand{\calB}{\mathcal{B}}
\newcommand{\calS}{\mathcal{S}}
\newcommand{\calN}{\mathcal{N}}
\DeclareMathOperator{\diag}{diag}
\newcommand{\eps}{\varepsilon}
\newcommand{\bfX}{\boldsymbol{X}}
\newcommand{\bfY}{\boldsymbol{Y}}

\newcommand{\SSDE}{\text{SSDE}}
\DeclareMathOperator{\sgn}{sgn}
\DeclareMathOperator{\add}{add}
\DeclareMathOperator{\Res}{Res}
\newcommand{\bfz}{\boldsymbol{z}}
\newcommand{\bfc}{\boldsymbol{c}}
\newcommand{\bfx}{\boldsymbol{x}}
\newcommand{\bfy}{\boldsymbol{y}}
\newcommand{\bfs}{\boldsymbol{s}}

\newcommand{\qh}{\{\frac q2\}}
\newcommand{\mqh}{\{-\frac q2\}}
\newcommand{\qhp}{\{\frac q2+1\}}
\newcommand{\mqhm}{\{-\frac q2-1\}}
\newcommand{\qhm}{\{\frac q2-1\}}
\newcommand{\mqhp}{\{-\frac q2+1\}}
\renewcommand{\MR}[1]{}

\newif{\ifoverfulboxhacks}
\overfulboxhackstrue
\ifoverfulboxhacks
\newcommand{\OBHnotag}{\notag}
\newenvironment{equationaligned}{\align}{\endalign}
\else
\newcommand{\OBHnotag}{}
\newenvironment{equationaligned}{\equation\aligned}{\endaligned\endequation}
\fi

\title{Analysis of Carries in Signed Digit Expansions}
\author{Clemens Heuberger}

\author{Sara  Kropf}
\address[C.~Heuberger and S.~Kropf]{Institut f\"ur Mathematik, Alpen-Adria-Universit\"at Klagenfurt,
  Universit\"atsstra\ss e 65--67, 9020 Klagenfurt, Austria}
\email{\href{mailto:clemens.heuberger@aau.at}{clemens.heuberger@aau.at}}
\email{\href{mailto:sara.kropf@aau.at}{sara.kropf@aau.at}}
\thanks{C.~Heuberger and S.~Kropf are supported by the Austrian Science Fund (FWF):
  P~24644-N26 and by the Karl Popper Kolleg ``Modeling-Simulation-Optimization'' funded by the Alpen-Adria-Universität Klagenfurt and by the Carinthian Economic Promotion Fund (KWF)}

\author{Helmut Prodinger}
\address[H.~Prodinger]{Department of Mathematical Sciences, Stellenbosch University, 7602 Stellenbosch,
 South Africa}
\thanks{H.~Prodinger is supported by an incentive grant of the National
  Research Foundation of South Africa.}
\thanks{Parts of the
  article were written while S.~Kropf was a visitor at
  Stellenbosch University.}
\email{\href{mailto:hproding@sun.ac.za}{hproding@sun.ac.za}}

\keywords{Carry, central limit theorem, transducer, probabilistic automaton,
  symmetric signed digit expansion, von Neumann's addition}
\subjclass[2010]{11A63\mscexp{Radix representation; digital problems [For
    metric results, see 11K16]};
60C05\mscexp{Combinatorial probability},
60F05\mscexp{Central limit and other weak theorems},
68Q45\mscexp{Formal languages and automata [See also 03D05, 68Q70, 94A45]},
68W40\mscexp{Analysis of algorithms [See also 68Q25]}}

\begin{document}
\begin{abstract}
  The number of positive and negative carries in the addition of two
  independent random signed digit expansions of given length is analyzed
  asymptotically for the $(q, d)$-system and the symmetric signed digit
  expansion. The results include expectation, variance, covariance between the
  positive and negative carries and a central limit theorem.

  Dependencies between the digits require determining suitable transition
  probabilities to obtain equidistribution on all expansions of given length.
  A general procedure is described to obtain such transition probabilities for
  arbitrary regular languages.

  The number of iterations in von Neumann's parallel addition method
  for the symmetric signed digit expansion is also analyzed, again including
  expectation, variance and convergence to a double exponential limiting
  distribution. This analysis is carried out in a general framework for
  sequences of generating functions.
\end{abstract}  
\maketitle

\section{Introduction}

Addition is an essential arithmetic operation in many algorithms. As the efficiency of
addition is influenced by the number of occurring carries, we asymptotically
analyze this number, which depends on the base and the digit set of the digit expansion.

We consider two different types of digit expansions: On the one hand, we investigate
$(q,d)$-expansions, that are extensions of the standard $q$-ary digit
expansion with digit set $\{d,\allowbreak\ldots,\allowbreak q+d-1\}$. With $d=0$, this includes the
case of the standard $q$-ary expansion. Consecutive
digits are independent in this case. On the other hand, the symmetric signed
digit expansion \cite{Heuberger-Prodinger:2001:minim-expan} has an even base $q$ and the redundant digit set
$\{-q/2,\ldots,q/2\}$. To remove the redundancy, there is a syntactical rule
to decide which of the digits $-q/2$ and $q/2$ is used. This rule introduces
dependencies between consecutive digits.

Two different addition
algorithms are investigated. The first one is the standard addition:  We add
 two digits starting at the least significant position. If the result is not in the given digit
set or does not fulfill the syntactical conditions, then a non-zero carry is produced. This
carry is added to the sum of the two digits at the next position. An example for this standard addition of two
decimal expansions is given in Table~\ref{tab:example-standard-addition}.

\begin{table}
  \centering
  \begin{equation*}
    \begin{array}{c@{\,}c@{\,}c@{\,}c}
      2_{\phantom{0}}&1_{\phantom{1}}&{4}_{\phantom{1}}&{6}_{\phantom{0}}\\
      1_{0}&{2}_{1}&{5}_{1}&{5}_{0}\\\hline
      3_{\phantom{0}}&4_{\phantom{1}}&{0}_{\phantom{1}}&{1}_{\phantom{0}}
    \end{array}
  \end{equation*}
  \caption{Example for standard addition in the decimal system. The
    subscripts in the second row are the carries.}
  \label{tab:example-standard-addition}
\end{table}

In the case of positive and negative digits, positive and negative carries occur.
The parameters of interest is their number for an independent pair of random summands of given length.

In contrast to standard addition, von Neumann's addition is a parallel
algorithm with several iterations. The idea is to add the digits at each position
in parallel (the interim result). If this result is not admissible in the given digit system, then a non-zero carry
is produced and the interim result is corrected correspondingly at this position. However, this carry is not
added immediately: The interim result
and the carries are the input for the next iteration. When the carry sequence
only contains zeros, then the algorithm terminates. An example for von
Neumann's addition is shown in Table~\ref{tab:intro-ssde-10} for the addition of two decimal
expansions. 
\begin{table}
  \begin{center} \setlength{\extrarowheight}{1pt}
$$    \begin{array}{rl}
      5377& \text{first summand}\\
      8125& \text{second summand}\\\hline
      3492& \text{first interim result}\\
      10010&\text{carries}\\\hline
      13402&\text{second interim result}\\
      000100&\text{carries}\\\hline
      013502&\text{final result}\\
      0000000&
      \end{array}$$
    \caption{Example for von Neumann's addition in the decimal system.}
    \label{tab:intro-ssde-10}
  \end{center}
\end{table}

The number of iterations of von Neumann's addition is of
interest as it corresponds to the running time. 

Diaconis and Fulman \cite{Diaconis-Fulman:2014:combin} and Nakano and Sadahiro \cite{Nakano-Sadahiro:2014:carries} consider the carries of the standard addition as a Markov chain. This is only valid if the  digits of the digit
expansion are independent. In their analysis, they obtain a stationary distribution.
In this article, we determine the expectation, variance and central limit theorem for the number of positive and negative carries as well as the covariance between the positive and negative carries in the $(q,d)$-system and the symmetric signed digit system. The authors of \cite{Diaconis-Fulman:2014:combin} concentrate on an \emph{odd} basis $q$ and the symmetric digit set $\{-(q-1)/2,\ldots,(q-1)/2\}$. The symmetric signed digit expansion (defined later)
is the natural way to define a unique representation with a symmetric set of digits and an \emph{even} base $q$.
Thus, a part of the present paper can be seen as a complement of \cite{Diaconis-Fulman:2014:combin}.

The expected number of iterations of von Neumann's addition was analyzed in \cite{Knuth:1978:carry} and \cite{Heuberger-Prodinger:2003:carry-propag} for standard
$q$-ary expansions and $(q,d)$-expansions, respectively. It turns out that the expected number of iterations is logarithmic in the length of the expansions.
In
\cite{Heuberger-Prodinger:2003:carry-propag}, symmetric signed digit
expansions are analyzed, too, but with a simplified probabilistic model since a
precise probabilistic model exceeded computing resources available at that
time. This simplification has a significant influence on the main term. In
this paper, we combine advances in soft- and hardware with sophisticated use
of the finite state machine package of
SageMath~\cite{SageMath:2016:7.0} to tackle the precise model
in roughly 10 minutes of CPU time. The results include expectation, variance and convergence to a double exponential distribution.

The outline of the paper is as follows. In Section~\ref{sec:digit-exp}, we define $(q,d)$-expansions and symmetric
signed digit expansions. We first analyze the standard addition in Sections~\ref{sec:standard-addition}--\ref{sec:asympt-analys-stand}. The algorithms and the corresponding transducers
for the standard addition of $(q,d)$-expansions and symmetric signed digit
expansions are presented in Section~\ref{sec:standard-addition}.  Our probabilistic model
is to choose both summands of length $\ell$ independently such that each
expansion of length $\ell$ is equally likely. In the case of the symmetric
signed digit expansions, the dependencies between the digits require
approximating the equidistribution with an error that does not influence the final result. The corresponding probabilities are defined in Lemma~\ref{lemma:conditional-probabilities} in Section~\ref{sec:approx-equidist} for general
regular languages, see also \cite{Shannon:1948:mathem-theor-commun} and \cite{Parry:1964:intrin-markov}. In Section~\ref{sec:asympt-analys-stand}, we combine this approximate
equidistribution with the transducers from Section~\ref{sec:automata} to
obtain  an asymptotic analysis including the expectation, the
variance and asymptotic normality in the main Theorems \ref{thm:q-d-stand} and \ref{thm:ssde-stand} for the $(q, d)$-system and the symmetric signed digit system, respectively. 

Then, we analyze von Neumann's addition. We start in
Section~\ref{sec:von-neum-addit} with the algorithms and the automaton. Theorem~\ref{theorem:bootstrapping} provides a general framework for the analysis of sequences occurring in this context. Then we again use the approximate equidistribution from
Section~\ref{sec:approx-equidist} to asymptotically analyze the number of
iterations of von Neumann's addition in
Theorem~\ref{thm:ssde-neumann} in Section~\ref{sec:asympt-analys-von}. This
analysis extends the results in
\cite{Knuth:1978:carry} and \cite{Heuberger-Prodinger:2003:carry-propag} to the
symmetric signed digit expansions and to include not only the expected value
but also the variance and a convergence in distribution.

Obtaining the values of the constants occuring in the asymptotic analysis of standard and von Neumann's addition requires computations 
involving finite state machines and determinants of matrices in several variables. These
computations are performed using the mathematical software system SageMath
\cite{SageMath:2016:7.0}. Notebooks containing all the
computations can be found at
\cite{Heuberger-Kropf-Prodinger:2015:analy-carries-online}. However, the
existence of these constants follows from the theoretical results.

\section{Digit Expansions}\label{sec:digit-exp}
In this section, we define the digit expansions which will be used in later
sections. We also recall their properties.

\subsection{$(q,d)$-expansions}
\begin{definition}
  Let $-q<d\leq 0$ be two integers with $q\geq 2$. The $(q,d)$-expansion of an integer $x$ is the
  $q$-ary expansion $(x_{\ell}\ldots x_{0})_{q}$ with digits
  $x_{i}\in\{d,\ldots, q+d-1\}$ such that $x=\sum_{i=0}^{\ell}x_{i}q^{i}$.
\end{definition}

\begin{example}
  The $(4,-1)$-expansion of $3$ is $(1\bar{1})_{4}$, where we write $\bar{1}$
  for the digit $-1$.
\end{example}

The $(q,d)$-expansion exists for all integers if $d\neq 0$ and
$d\neq -q+1$. For $d=0$ (this is the standard $q$-ary expansion), only the non-negative integers have a
$(q,d)$-expansion. Conversely, for $d=-q+1$, only  the non-positive integers
have a $(q,d)$-expansion. If the $(q,d)$-expansion of an integer exists, then
it is unique up to leading zeros.

If $q$ is odd and $d=\frac{-q+1}{2}$, then the $(q,d)$-expansion minimizes the
sum of absolute values of the digits among all $q$-ary expansions with
arbitrary digits (see~\cite{Heuberger-Prodinger:2001:minim-expan}).

\subsection{Symmetric Signed Digit Expansion}

We recall the definition of the symmetric signed digit expansion (SSDE) as defined in~\cite{Heuberger-Prodinger:2001:minim-expan} and further analyzed
in \cite{Heuberger-Prodinger:2003:carry-propag}.

\begin{definition}
  Let $q\geq 2$ be an even integer.
  The symmetric signed digit expansion (SSDE) of an integer is the $q$-ary
  digit expansion $(x_{\ell}\ldots x_{0})_{q}$ with $x_{i}\in\{-\frac
  q2,\ldots,\frac q2\}$ such that the syntactical rule 
  \begin{equation*}
    \lvert x_{j}\rvert=\frac q2\quad \Longrightarrow\quad 0\leq\sgn(x_{j})x_{j+1}\leq \frac q2-1
  \end{equation*}
  is satisfied for $0\leq j<\ell$.
\end{definition}
In \cite{Heuberger-Prodinger:2001:minim-expan}, it is shown that each integer $n$
has a unique SSDE (up to leading zeros). It minimizes the sum of absolute values of the digits among all $q$-ary
expansions of $n$ with arbitrary digits (cf.\ \cite{Heuberger-Prodinger:2001:minim-expan}).

For $q=2$, we obtain the digit set $\{0,\pm1\}$ and the syntactical rule that
at least one of any two adjacent digits is zero. This digit expansion is also
called non-adjacent form (cf.\ \cite{Reitwiesner:1960}).

\section{Standard Addition}\label{sec:standard-addition}
We write bold face letters for
sequences which are padded with zeros on the left.

Let $\bfx=\ldots x_{1}x_{0}$
and $\bfy=\ldots y_{1}y_{0}$ be the two summands given as $q$-ary expansions with
digit set $D$
(possibly satisfying some syntactical rules). Then standard addition can be
written in the form
\begin{equation*}
  \begin{array}{cccc}
    \ldots &&{x_{1}}_{\phantom{c_{1}}}&{x_{0}}_{\phantom{c_{0}}}\\
    \ldots &{}_{c_{2}}&{y_{1}}_{c_{1}}&{y_{0}}_{c_{0}}\\\hline
    \ldots &&{z_{1}}_{\phantom{c_{1}}}&{z_{0}}_{\phantom{c_{0}}}
  \end{array}
\end{equation*}
where $x_{i}+y_{i}+c_{i}=z_{i}-qc_{i+1}$, $c_{0}=0$ with $z_{i}\in D$ and $\bfz=\ldots
z_{1}z_{0}$ satisfying the syntactical rules of the digit system under consideration. We
asymptotically analyze the
sequence of carries $\bfc=\ldots c_{2}c_{1}$.

From a different point of view, the standard addition with digit
set $D$ is a conversion between different digit sets: We have a $q$-ary
digit expansion with digits in $D+D$ and we want to transform this
digit expansion into a digit expansion with digit set $D$ satisfying all
syntactical rules. This can be written in the form
\begin{equation*}
  \begin{array}{cccc}
    \ldots &{}_{c_{2}}&{s_{1}}_{c_{1}}&{s_{0}}_{c_{0}}\\\hline
    \ldots &&{z_{1}}_{\phantom{c_{1}}}&{z_{0}}_{\phantom{c_{0}}}
  \end{array}
\end{equation*}
where $s_{i}=x_{i}+y_{i}\in D+D$. We call the sequence $\bfs$ the digitwise
sum of $\bfx$ and $\bfy$ and write $\bfs=\bfx+\bfy$.

We will mostly use this point of view. Most of the algorithms and
  transducers require the input of $\bfs$. If there are
  syntactical rules for $\bfx$ and $\bfy$, then the
  sequence $\bfs$ can not be arbitrary.
\begin{remark}\label{rem:interchange}
  From this point of view, it is clear that interchanging two digits $x_{i}$ and $y_{i}$ of the two summands
  does not influence the result, but only both summands. The carries, the
  digitwise sum and the steps taken by the
  algorithms and the transducers stay the same as they depend only on the
  digitwise sum.
\end{remark}

\subsection{Algorithms}
\subsubsection{Standard Addition for $(q,d)$-expansions}The digit set is
$D=\{d,\ldots, q+d-1\}$.
Algorithm~\ref{alg:stand-add-q-d} transforms a $q$-ary
expansion with digit set $D+D$ into a $(q,d)$-expansion. As there are no
syntactical rules, all digits are
independent. Thus, we do not have to look ahead when choosing the carry.

An example of standard addition for $(5,-1)$-expansions using this algorithm
is given in Table~\ref{tab:std-q-d-ex}.

\begin{table}
  \centering
  \begin{equation*}
    \setlength{\extrarowheight}{2pt}
    \begin{array}{c@{\,}c@{\,}c@{\,}c}
      1_{\phantom{0}}&2_{\phantom{1}}&3_{\phantom{1}}&{\bar 1}_{\phantom{0}}\\
      1_{1}&{2}_{0}&{1}_{\bar 1}&{\bar 1}_{0}\\\hline
      {3}_{\phantom{1}}&{\bar 1}_{\phantom{1}}&{3}_{\phantom{0}}&3_{\phantom{0}}
    \end{array}
  \end{equation*}
  \caption{Example for standard addition for $(5,-1)$-expansions. The subscripts in the second row are the carries.}
  \label{tab:std-q-d-ex}
\end{table}

\begin{algorithm}
  \caption{Standard addition for two $(q,d)$-expansions }
  \label{alg:stand-add-q-d}
  \begin{algorithmic}
    \Require digit expansion $(s_{\ell}\ldots s_{0})_{q}$ with digits in $\{2d,\ldots,2q+2d-2\}$ 
    \Ensure $(q,d)$-expansion $z$ of $(s_{\ell}\ldots s_{0})_{q}$
    \State $z=()$
    \State $c=0$
    \For{$j=0$ to $\ell$}
        \State $a=s_{j}+c$
        \State $c=0$
        \If{$a\geq q+d$}      
          \State $c=1$
        \ElsIf{$a\leq d-1$}
          \State $c=-1$
        \EndIf
        \State $a=a-cq$
        \State $z = (a)+z$
    \EndFor
  \end{algorithmic}
\end{algorithm}

\subsubsection{Standard Addition for SSDEs}
Let $q\geq 2$ be even. Algorithm~\ref{alg:stand-add-ssde} transforms a $q$-ary
expansion with digit set $\{-q,\ldots,q\}$ into a SSDE. As the choice between
the redundant digits $\frac q2$ and $-\frac q2$ depends on the next digit, we
have to look ahead at the next digit in these cases. This algorithm is an
extension of the one in \cite{Heuberger-Prodinger:2003:carry-propag} taking
into account that we start with a larger digit set.

An example of standard addition for SSDEs with $q=4$ using this algorithm
is given in Table~\ref{tab:std-ssde-ex}.

\begin{table}
  \centering
  \begin{equation*}
    \setlength{\extrarowheight}{2pt}
        \begin{array}{c@{\,}c@{\,}c@{\,}c}
          1_{\phantom{0}}&{\bar1}_{\phantom{1}}&0_{\phantom{1}}&{2}_{\phantom{0}}\\
          1_{0}&{\bar 1}_{1}&{1}_{1}&{2}_{0}\\\hline
          {2}_{\phantom{1}}&{\bar 1}_{\phantom{1}}&{\bar2}_{\phantom{0}}&0_{\phantom{0}}
        \end{array}
  \end{equation*}
  \caption{Example for standard addition for SSDEs for $q=4$. The subscripts in the second row are the carries.}
  \label{tab:std-ssde-ex}
\end{table}

\begin{algorithm}
  \caption{Standard addition for two SSDEs}
  \label{alg:stand-add-ssde}
  \begin{algorithmic}
    \Require digit expansion $(s_{\ell}\ldots s_{0})_{q}$ with digits in $\{-q,\ldots,q\}$ 
    \Ensure SSDE $z$ of $(s_{\ell}\ldots s_{0})_{q}$
    \State $s_{\ell+1}=0$
    \State $z=()$
    \State $c=0$
    \For{$j=0$ to $\ell$}
        \State $a=s_{j}+c$
        \State $c=0$
        \If{$a>\frac q2$}
            \State $c=1$
        \ElsIf{$a<-\frac q2$}
            \State $c=-1$
        \ElsIf{$a=\frac q2$ and ($-\frac q2\leq s_{j+1}<0$ or $\frac q2\leq s_{j+1}<q$)}
            \State $c=1$
        \ElsIf{$a=-\frac q2$ and ($-q<s_{j+1}\leq -\frac q2$ or $0<
          s_{j+1}\leq \frac q2$)}
            \State $c=-1$
        \EndIf
        \State $a=a-cq$
        \State $z = (a)+z$
    \EndFor
  \end{algorithmic}
\end{algorithm}

\subsection{Transducers}\label{sec:automata}
In this section, we present the transducers for the algorithms
presented in the last section.

We are not interested in the output of the addition, but only in the
carries. Thus we only use the carries as the output of the transducer. But, if
required, the output digits can easily be reconstructed. 

In our setting, a transducer consists of a finite set of states $S$, a finite input
alphabet $D+D$, an output alphabet, a set of transitions $E\subseteq
S^{2}\times (D+D)$ with input labels in $D+D$, output labels in the output alphabet for each
transition, and an initial state. All states are final.

The input of the transducer is a digit expansion with digits in $D+D$. The
output of the transducer is the sequence of labels of a path starting in the
initial state with the given input as the input label. In our cases, there exists
always such a path and it is unique (i.e., the transducer is complete and deterministic).

The labels of the states encode the current carry (except for the situations
when we have to look ahead). The number of states is independent of $q$. The number of transitions between two states
depends on the base $q$.

To plot the transducer, we group these transitions
and their labels. We draw only one arc and write the label $M\mid c$ for
a set $M\subset D+D$ to represent a group of transitions consisting of one transition  with input label $m$ and
output label $c$ for every $m\in M$.  If $M$ is the empty set, then there are
no such
transitions. This may happen for special values of $d$ or $q$.

 The output label of a transition is one carry $c$, a pair of carries $c$, or no carry $c$, i.e., $c\in\{0,1,\bar1,-\}\cup\{0,1,\bar1\}^{2}$, where ``$-$''
denotes the empty output.
The input of the transducer is the sequence $\bfs$ of digitwise
sums.

Let $\ell$ and $u$ be the minimum and the maximum of the extended digit set $D+D$. For the labels of the transitions, we define 
\begin{align*}
M+\eps&=\big(\{m+\eps\mid m\in M\}\cup(M\cap \{\ell,u\})\big)\cap(D+D)
\end{align*}
for $\eps=\pm1$ and a set $M$. This definition is motivated by the following
interpretation: Whenever a set $M=\{j, \ldots, u\}$ occurs, it is actually
meant to be the interval $[j, \infty)$ intersected with the extended digit
set. Subtracting $1$ leads to $[j-1, \infty)$, again intersected with
the extended digit set. This corresponds to $M-1$ as defined above.

\subsubsection{Standard Addition for $(q,d)$-expansions}
The transducer in Figure~\ref{fig:stand-add-q-d} computes the carries as in
Algorithm~\ref{alg:stand-add-q-d}. We use the sets $L=\{2d,\ldots,d-1\}$,
$D=\{d,\ldots,q+d-1\}$ and $H=\{q+d,\ldots,2q+2d-2\}$. 

The transitions are
constructed by using Algorithm~\ref{fig:stand-add-q-d} for the current input
and carry.

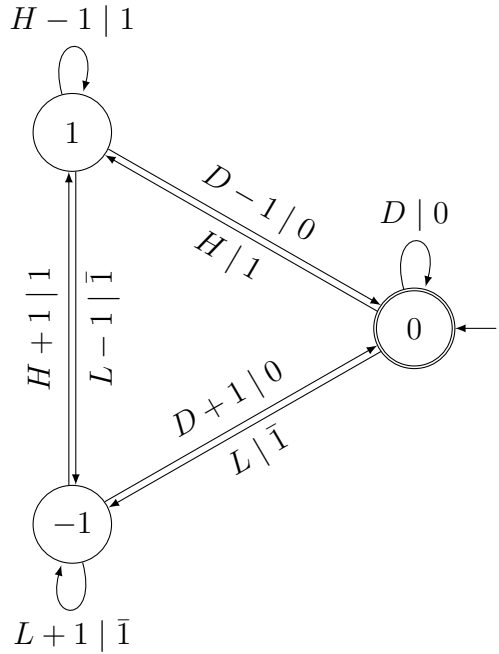
\begin{figure}
\centering
\begin{tikzpicture}[auto, initial text=, >=latex]
\node[state, accepting, initial, initial where=right] (v0) at (3.000000, 0.000000) {$0$};
\node[state] (v1) at (-1.500000, 2.598076) {$1$};
\node[state] (v2) at (-1.500000, -2.598076) {$-1$};
\path[->] (v2) edge[loop below] node {$L+1\mid \bar1$} ();
\path[->] (v2.35.00) edge node[rotate=30.00, anchor=south] {$D+1\mid
0$} (v0.205.00);
\path[->] (v2.95.00) edge node[rotate=90.00, anchor=south] {$H+1\mid
1$} (v1.265.00);
\path[->] (v0.-145.00) edge node[rotate=30.00, anchor=north]
{$L\mid \bar1$} (v2.25.00);
\path[->] (v0) edge[loop above] node {$D\mid 0$}
();
\path[->] (v0.155.00) edge node[rotate=330.00, anchor=north] {$H\mid 1$} (v1.325.00);
\path[->] (v1.-85.00) edge node[rotate=90.00, anchor=north] {$L-1\mid
\bar1$} (v2.85.00);
\path[->] (v1.-25.00) edge node[rotate=-30.00, anchor=south] {$D-1\mid
0$} (v0.145.00);
\path[->] (v1) edge[loop above] node {$H-1\mid 1$} ();
\end{tikzpicture}
\caption{Standard addition for two $(q,d)$-expansions.}
\label{fig:stand-add-q-d}
\end{figure}

\subsubsection{Standard Addition for SSDEs}
The transducer in Figure~\ref{fig:stand-add-ssde} computes the carries as in
Algorithm~\ref{alg:stand-add-ssde}. We use the sets $L=\{0,\ldots,\frac
q2-1\}$, $H=\{\frac q2+1,\ldots,q\}$  and
$H_{q}=\{\frac q2,\ldots,q-1\}$.

The transitions are
constructed by using Algorithm~\ref{fig:stand-add-ssde} for the current input
and carry.

The labels of the states $-1$, $0$ and $1$ encode the current carry. In the
states with labels $\pm\frac q2$, we do not know yet whether the digit of the
sum should be $\frac q2$ or $-\frac q2$ and thus, which carry is produced. To
decide this, we have to look at the next digit. Thus, the transitions leading
to a state $\pm\frac q2$ have no output ($-$) and the transitions starting at
a state $\pm\frac q2$ have two output digits.

\begin{figure}
\centering
\begin{tikzpicture}[auto, initial text=, >=latex, scale=0.85]
\path[use as bounding box] (-9.7,-8.5) rectangle (7,7);
\node[state, accepting, initial, initial where=right] (v0) at (2*3.000000, 0.000000) {$0$};
\node[state] (v1) at (-2*1.500000, 2*2.598076) {$1$};
\node[state] (v2) at (-2*1.500000, -2*2.598076) {$-1$};
\node[state] (v3) at (0, 0) {$\frac q2$};
\node[state] (v4) at (-2*4.5, -2*2.598076) {$-\frac q2$};
\path[->] (v2) edge[loop right] node[rotate=0] {$-H+1\mid \bar1$} ();
\path[->] (v2.35.00) edge node[rotate=30.00, anchor=south] {$(-L\cup L)+1\mid0$} (v0.205.00);
\path[->] (v2.95.00) edge node[rotate=90.00, anchor=south] {$H+1\mid1$} (v1.265.00);
\path[->] (v0.-145.00) edge node[rotate=30.00, anchor=north] {$-H\mid\bar1$} (v2.25.00);
\path[->] (v0) edge[loop above] node {$-L\cup L\mid0$} ();
\path[->] (v0.155.00) edge node[rotate=330.00, anchor=north] {$H\mid1$} (v1.325.00);
\path[->] (v1.-85.00) edge node[rotate=90.00, anchor=north] {$-H-1\mid\bar1$} (v2.85.00);
\path[->] (v1.-25.00) edge node[rotate=-30.00, anchor=south] {$(-L\cup L)-1\mid0$} (v0.145.00);
\path[->] (v1) edge[loop above] node {$H-1\mid1$} ();
\path[->] (v3.5.00) edge node[rotate=0.00, anchor=south] {$L\mid00$, $-L-1\mid01$} (v0.175.00);
\path[->] (v3.125.00) edge node[rotate=-60.00, anchor=north]
{$H_{q}\mid11$, $\{q\}\mid10$} (v1.-65.00);
\path[->] (v3.-115.00) edge node[rotate=60.00, anchor=north] {$-H\mid\bar10$} (v2.55.00);
\path[->] (v0.185.00) edge node[rotate=0.00, anchor=north] {$\qh\mid -$} (v3.-5.00);
\path[->] (v1.-55.00) edge node[rotate=-60.00, anchor=south] {$\qhm\mid -$} (v3.115.00);
\path[->] (v2.65.00) edge node[rotate=60.00, anchor=south] {$\qhp\mid -$} (v3.-125.00);
\path[->] (v4.-55.00) edge[bend right=80] node[rotate=20.00, anchor=south] {$L+1\mid0\bar1$, $-L\mid00$} (v0.-95.00);
\path[->] (v4.65.00) edge node[rotate=60.00, anchor=south] {$H\mid10$} (v1.-125.00);
\path[->] (v4.5.00) edge node[rotate=0.00, anchor=south]
{$-H_{q}\mid\bar1\bar1$, $\{-q\}\mid\bar10$} (v2.175.00);
\path[->] (v0.-85.00) edge[bend left=85] node[rotate=20.00, anchor=north]
{$\mqh\mid -$} (v4.-65.00);
\path[->] (v1.-115.00) edge node[rotate=60.00, anchor=north] {$\mqhm\mid -$} (v4.55.00);
\path[->] (v2.185.00) edge node[rotate=00.00, anchor=north] {$\mqhp\mid -$}
(v4.-5.00);
\end{tikzpicture}
\caption{Standard addition for two SSDEs.}
\label{fig:stand-add-ssde}
\end{figure}
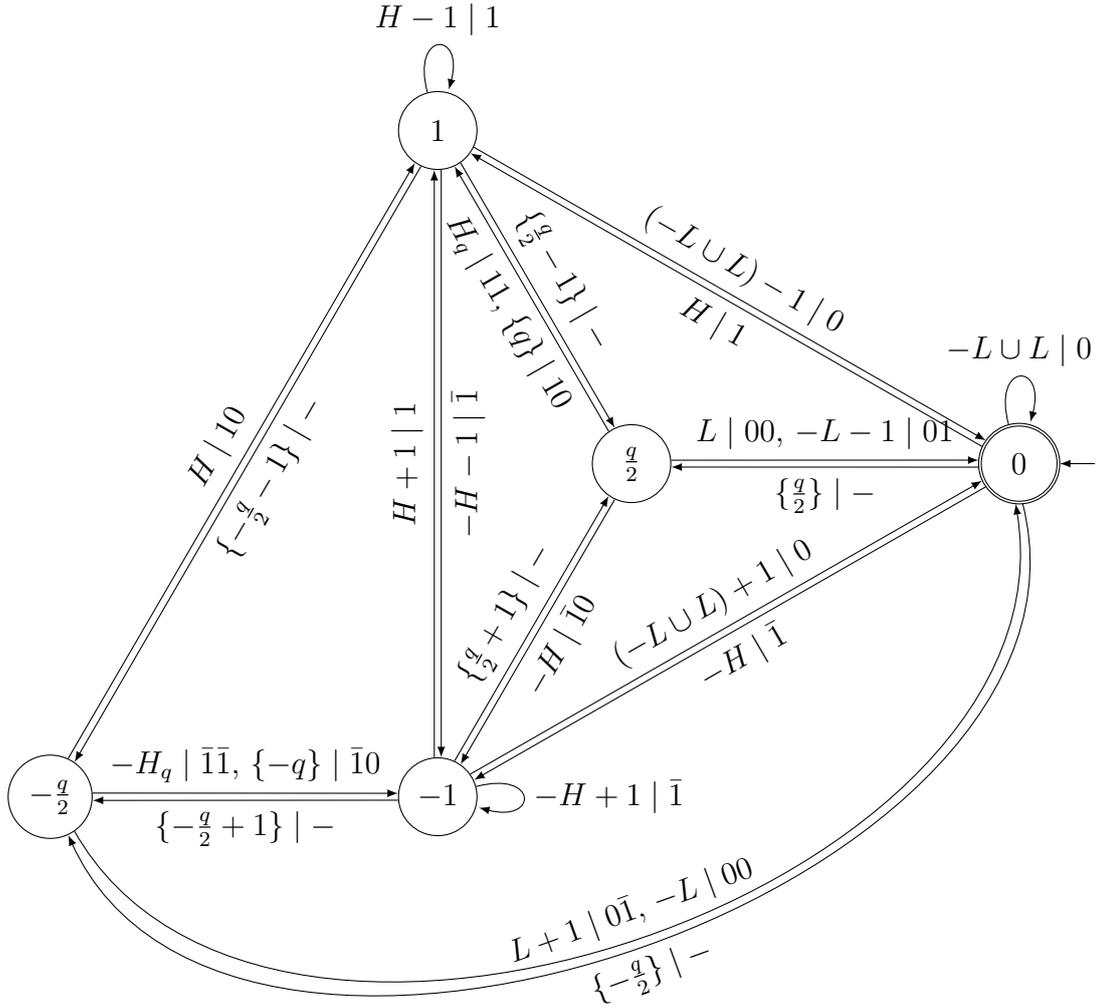

\section{Approximate Equidistribution}\label{sec:approx-equidist}
As a probabilistic input model, we want to use an equidistribution on all
digit expansions satisfying certain syntactical rules. This is easy in the
case of $(q,d)$-expansions (see Section~\ref{sec:weights-q-d}) because there
are no syntactical rules. But in the case of a general regular language, like the
SSDE, we can only approximate an equidistribution by
Lemma~\ref{lemma:conditional-probabilities}. However, this approximation does
not influence the main terms of the results. 

A regular language is recognized by an automaton. An automaton is defined to
consist of states, transitions between these states with labels, an initial
state and final states. So to say, it is a transducer without output. The
automaton recognizes a word from a language, if there exists a path starting
at the initial state, leading to a final state with this word as the label. 

We call an automaton
aperiodic if its underlying directed graph is aperiodic, i.e., the greatest common
divisor of all lengths of directed cycles of the graph is $1$. If the underlying
directed graph
is strongly connected, then the automaton is so, too. If an automaton is
strongly connected and aperiodic, then the adjacency matrix of the underlying
graph is primitive.

Given an automaton $\calA$ for a regular language, we automatically construct 
transition probabilities between the states to obtain an approximate equidistribution on
all words of given length $\ell$. The weight of the word is the product of the
transition probabilities multiplied with an \emph{exit weight} (the factor in front of
the product in \eqref{eq:weight-of-a-path} below). This corresponds to an
approximate equidistribution on
all paths of length $\ell$ of the underlying graph of the automaton starting
in the initial state. Without
the exit weights, these transition probabilities are the same as defined by Shannon in
\cite{Shannon:1948:mathem-theor-commun} and Parry in \cite{Parry:1964:intrin-markov}.
We implemented the computations of this lemma as part of
SageMath~\cite{SageMath:2016:7.0} as \texttt{Automaton.shannon\_parry\_markov\_chain}.

\begin{lemma}\label{lemma:conditional-probabilities}
  Let $\calA$ be a deterministic automaton with set of states $\{1,\ldots, n\}$,
  initial state $1$, final states $\emptyset\neq F\subseteq \{1,\ldots, n\}$ recognizing a
  regular language $\calL$. 
  We assume that the adjacency matrix $A$ of the underlying graph of $\calA$ is primitive.

  The dominant eigenvalue of $A$ is denoted by $\lambda$, all other eigenvalues
  of $A$ are assumed to be of modulus less than or equal to $\xi\lambda$ for
  some $0<\xi<1$. If there are eigenvalues of modulus $\xi\lambda$,
  then each of them must be semisimple, i.e., its algebraic and geometric
  multiplicities coincide.

  Let $w>0$ and $u>0$ be right and left eigenvectors of $A$ for the eigenvalue $\lambda$,
  respectively, such that $w_1=1$ and $\langle u,w\rangle =1$.

  For a transition $t$ from some state $i$ to some state $j$, we set 
  \begin{equation}\label{eq:p_t_definition}
    p_{t}=\frac{w_j}{w_i \lambda}.
  \end{equation}

  For $\ell\ge 0$, the set of words of $\calL$ of length $\ell$ is denoted by $\calL_\ell$.
  For a word $x\in\calL_\ell$, we denote the states and
  transitions used when $\calA$ reads $x$ by $1=s_0$, \ldots, $s_\ell$ and 
  $t_1$, \ldots, $t_\ell$, respectively. The weight $W_{\ell}(x)$ of $x$ is then defined to be
  \begin{equation}\label{eq:weight-of-a-path}
    W_{\ell}(x)=\frac{1}{w_{s_\ell} \langle u, e_F\rangle}\prod_{j=1}^{\ell} p_{t_j}
  \end{equation}
  where $e_F$ is the indicator vector of the set $F$ of final states.

  Then 
  \begin{equation}\label{eq:stochastic-matrix}
    \sum_{t\text{ leaves }i}p_{t}=1
  \end{equation}
  holds for all states $i$ and
  \begin{equation}\label{eq:transition-probabilities-quasi-equidistribution}
    W_{\ell}(x)=\frac1{|\calL_\ell|}(1+O(\xi^\ell))
  \end{equation}
  holds uniformly for $\ell\ge 0$ and $x\in \calL_\ell$.

  Furthermore, consider the time-homogeneous Markov chain $\calM$ on the state space $\{1,\allowbreak
  \ldots,\allowbreak n\}$ where the transition probability from state $i$ to state $j$
  is $\sum_{t}p_t$ where the sum runs over all transitions in $\calA$ from $i$
  to $j$. Then this Markov chain has the stationary distribution
  \begin{equation}\label{eq:stationary-distribution}
    (u_1w_1, \ldots, u_nw_n).
  \end{equation}
\end{lemma}

For large $\ell$ and a transition $t$ from some state $i$ to some state $j$,
$p_t$ can be thought as the probability of using $t$ under the condition that
the automaton is currently in state $i$. Note that the sum in
\eqref{eq:stochastic-matrix} runs over all transitions leaving $i$ such that
multiple transitions between $i$ and $j$ are counted separately although their
individual weights $p_t$ only depend on $i$ and $j$. It turns out that the exit
weights do not influence
the main term of our asymptotic expressions. 

\begin{proof}[Proof of Lemma~\ref{lemma:conditional-probabilities}]
  We first note that the cardinality $|\calL_\ell|$ is given by
  \begin{equation*}
    |\calL_\ell|=e_1^\top A^\ell e_F = \langle e_1,w\rangle \langle
    u,e_F\rangle \lambda^{\ell}(1+O(\xi^\ell))=\langle
    u,e_F\rangle \lambda^{\ell}(1+O(\xi^\ell))
  \end{equation*}
  where $e_1=(1,0,\ldots, 0)$.

  For $x\in\calL_\ell$ with associated sequence of states $(s_0,\ldots,
  s_\ell)$, we have
  \begin{equation*}
    W_{\ell}(x)=\frac{1}{w_{s_\ell} \langle u, e_F\rangle}\prod_{j=1}^{\ell} 
    \frac{w_{s_j}}{w_{s_{j-1}} \lambda}
    =\frac{1}{w_{s_0} \langle u, e_F\rangle \lambda^\ell}.
  \end{equation*}
  As $w_{s_0}=w_1=1$, we get
  \eqref{eq:transition-probabilities-quasi-equidistribution}.

  Next, we prove \eqref{eq:stochastic-matrix} by rewriting the sum as
  \begin{equation*}
    \sum_{t\text{ leaves }i}p_{t}=\sum_{j=1}^{n} a_{ij}\frac{w_j}{w_i \lambda}=1
  \end{equation*}
  by definition of $w$.

  Finally, the transition matrix of the Markov chain $\calM$ is
  \begin{equation*}
    P=\Bigl(a_{ij}\frac{w_j}{w_i\lambda}\Bigr)_{1\le i,j\le n}
  \end{equation*}
  by definition of the Markov chain and \eqref{eq:p_t_definition}. Thus
  \begin{equation*}
    P=\frac1{\lambda}\diag\Bigl(\frac{1}{w_1}, \ldots, \frac{1}{w_n}\Bigr) A
    \diag(w_1, \ldots, w_n).
  \end{equation*}
  As
  \begin{align*}
    (u_1w_1, \ldots, u_n w_n)P&=
    \frac{1}{\lambda}(u_1, \ldots, u_n)A\diag(w_1, \ldots, w_n)\\
    &=
    (u_1, \ldots, u_n)\diag(w_1, \ldots, w_n)\\
    &=(u_1w_1, \ldots, u_n w_n),
  \end{align*}
  $(u_1w_1, \ldots u_n w_n)$ is a left eigenvector of $P$ to the eigenvalue
  $1$. By definition of $u$ and $w$, $\sum_{i=1}^n u_iw_i=1$.
  As $\calM$ is aperiodic and irreducible, $(u_1w_1, \ldots, u_nw_n)$ is the
  unique left eigenvector with this property and therefore the stationary distribution.
\end{proof}
The weight $W_{\ell}$ induces a measure
on the words of length $\ell$. The total measure of all words of
length $\ell$ is $1$ up to an exponentially small
error, thus it is a probability measure up to an exponentially small error. Each word has exactly the same
weight. If we see the transition probabilities as a part of the automaton, we
obtain a probabilistic automaton:

\begin{definition}
A probabilistic automaton is an automaton together with a map $p\colon t\mapsto
p_t$ from the set of transitions
to the interval $[0,1]$ such that
\begin{equation*}
  \sum_{t\text{ leaves }s}p_{t}=1
\end{equation*}
holds for all states $s$.
 We call $p_t$ the weight or the probability of the transition $t$. 
\end{definition}

\subsection{Weights for $(q,d)$-expansions}\label{sec:weights-q-d}
We can use Lemma~\ref{lemma:conditional-probabilities} in this case, too, but the digits of a $(q,d)$-expansion are independent of each other because there
are no syntactical rules involving more than one digit. Therefore we can directly
obtain equidistribution, not only approximating it. We first describe the
direct way and later, in Remark~\ref{rem:q-d-weight-via-lemma}, we consider
using Lemma~\ref{lemma:conditional-probabilities}.

For any digit
$x_{0}\in D$, we use the weights $W_{\ell}(x_{0})=\frac 1q$. The exit weight
is $1$.
By independence, we have the weight
\begin{equation*}
W_{\ell}(x)=\frac{1}{q^{\ell}}
\end{equation*}
for a
digit expansion $x$ of length $\ell$.
With this weight, we have an equidistribution of all $(q,d)$-expansions of
length $\ell$.
\begin{figure}
  \centering
  \begin{tikzpicture}[auto, initial text=, >=latex]
\node[state, accepting, initial, initial where=right] (v0) at (0.000000, 0.000000) {$0$};
\path[->] (v0) edge[loop above] node {$d,\ldots,q+d-1$} ();
\end{tikzpicture}
  \caption{Automaton recognizing $(q,d)$-expansions.}
  \label{aut:q-d-exp}
\end{figure}
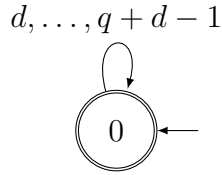

\begin{remark}\label{rem:q-d-weight-via-lemma}
  The same weights can be obtained by
  Lemma~\ref{lemma:conditional-probabilities}. The transition probabilities
  are $p_{0\rightarrow0}=q^{-1}$. As the automaton recognizing
  $(q,d)$-expansions has only one state (see Figure~\ref{aut:q-d-exp}), there is no error term in~\eqref{eq:transition-probabilities-quasi-equidistribution}.
\end{remark}

\subsection{Weights for SSDEs}\label{sec:weights-ssde}
The automaton in Figure~\ref{aut:ssde} recognizes SSDEs. The adjacency matrix
of this automaton is
\begin{equation*}
  A=\begin{pmatrix}0&\frac q2&0\\1&q-1&1\\0&\frac q2&0\end{pmatrix}
\end{equation*}
where the states are ordered by their labels.

The matrix $A$ has the eigenvalues $q$, $-1$ and $0$. The vectors $(\frac
1{q+1},\frac q{q+1},\frac 1{q+1})$
and $(\frac 12,1,\frac 12)^{\top}$ are the left and right
eigenvector corresponding to the eigenvalue $q$, respectively. The transition
probabilities are
\begin{alignat}{6}\label{eq:trans-prob}
  p_{-1\rightarrow0}&=p_{1\rightarrow 0}=\frac 2q,&\quad p_{0\rightarrow 1}&=p_{0\rightarrow-1}=\frac 1{2q},&\quad
  p_{0\rightarrow0}&=\frac 1q.
\end{alignat}
The constant in the error term is $\xi=\frac 1q$. The exit weights are $(2,1,2)\cdot \frac{q+1}{q+2}$.
\begin{figure}
  \centering
  \begin{tikzpicture}[auto, initial text=, >=latex]
    \node[state, initial, accepting, initial where=below] (v1) at (0, 0) {$0$};
    \node[state, accepting] (v2) at (4, 0) {$1$};
    \node[state, accepting] (v0) at (-4, 0) {$-1$};
    \path[->] (v1) edge[loop above] node {$-\frac q2+1,\ldots,\frac q2-1$} ();
    \path[->] (v0.5.00) edge node {$-\frac q2+1,\ldots,0$} (v1.175);
    \path[->] (v1.5) edge node {$\frac q2$} (v2.175);
    \path[->] (v2.185) edge node {$0,\ldots,\frac q2-1$} (v1.-5);
    \path[->] (v1.185) edge node {$-\frac q2$} (v0.-5);
  \end{tikzpicture}
  \caption{Automaton recognizing SSDEs.}
  \label{aut:ssde}
\end{figure}
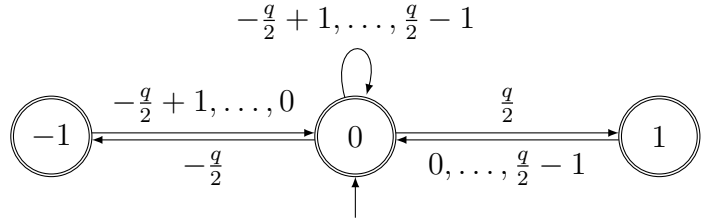

With these transition probabilities, the asymptotic frequencies of the digits
(cf.\ \cite{Heuberger-Prodinger:2001:minim-expan})
can be computed as
\begin{align}\label{eq:weights-digits-ssde}
  \begin{cases}
    \frac{1}{2(q+1)}&\text{if the digit is }\pm\frac{q}{2},\\
    \frac{q+2}{q(q+1)}&\text{if the digit is }0,\\
    \frac 1q&\text{otherwise}
  \end{cases}
\end{align}
by using the stationary distribution given in~\eqref{eq:stationary-distribution}.

\section{Asymptotic Analysis of the Standard Addition}\label{sec:asympt-analys-stand}
In this section, we use the probabilistic model defined in
Section~\ref{sec:approx-equidist} for the input sequence of the transducers in
Section~\ref{sec:automata}. Then we will use Lemma~\ref{lem:asy-analy} to obtain
expectation, variance and asymptotic normality of the number of
carries.

In Sections~\ref{sec:standard-addition-q-d} and \ref{sec:stand-addit-ssde}, we
will construct probabilistic automata whose transition labels are the carries and where
each transition has a weight corresponding to the weight constructed in
Section~\ref{sec:approx-equidist}.

Let $m$ and $n$ be two functions mapping the output of a transition into the
real numbers; for brevity we write $m(t)$ and $n(t)$ without mentioning the
output label of the transition $t$. In our setting $m$
  will count the number of carries $1$, and $n$ the number of carries $-1$
  of the output of a transition. We consider the two random variables
  $M_{\ell}$ and $N_{\ell}$ which are the sum of the values of $m$ and
  $n$, respectively, over a path of length $\ell$ with probability the product of the weights
  of this path multiplied with the exit weight.

  We will use multivariate generating functions in three variables $x$, $y$ and
  $z$. The variables $x$ and $y$ mark the number of carries $1$ and $-1$,
  respectively, and the variable $z$  marks the length of the expansion.

  The transition matrix $A(x,y)$ of a probabilistic automaton with $K$ states and two
  functions $m$ and $n$ is a $K\times K$
  matrix whose $(i,j)$-th entry is 
  \begin{equation*}
    \sum_{t\colon i\rightarrow j}p_{t}x^{m(t)}y^{n(t)}
  \end{equation*}
  where $p_{t}$ is the weight of the transition $t$. 

The next lemma is a slight modification of
\cite[Theorem~3.9]{Heuberger-Kropf-Wagner:2014:combin-charac} taking into
account the non-uniform distribution of the input alphabet.
\begin{lemma}\label{lem:asy-analy}
  Let $\mathcal A$ be a strongly connected, aperiodic probabilistic automaton where
  all states are final. Let $m$ and $n$ be functions mapping the output of a
  transition into the real numbers and $A(x,y)$ be the associated transition matrix of the automaton,
  where $x$ and $y$ mark $m$ and $n$, respectively. Let $M_\ell$ and $N_\ell$
  be the associated random variables as defined above.

  Define the function $f(x,y,z)=\det (I-zA(x,y))$.  Then the expected value of $(M_{\ell},N_{\ell})$ is $(e_{m},
  e_{n})\ell+\bigOh(1)$ with
  \begin{align*}
    e_{m}&=\left.\frac{f_{x}}{f_{z}}\right\vert_{(1,1,1)},\\
    e_{n}&=\left.\frac{f_{y}}{f_{z}}\right\vert_{(1,1,1)}.
  \end{align*}
  The variance-covariance matrix  $\big(\begin{smallmatrix}v_{m}&c\\c&v_{n}\end{smallmatrix}\big)\ell+\bigOh(1)$ has the entries
  \begin{align}\label{eq:var-std-add}
    v_{m}=\left.\frac{f_{x}^{2}(f_{zz}+f_{z})+f_{z}^{2}(f_{xx}+f_{x})-2f_{x}f_{z}f_{xz}}{f_{z}^{3}}\right\vert_{(1,1,1)},\\
    v_{n}=\left.\frac{f_{y}^{2}(f_{zz}+f_{z})+f_{z}^{2}(f_{yy}+f_{y})-2f_{y}f_{z}f_{yz}}{f_{z}^{3}}\right\vert_{(1,1,1)},\\
    c=\left.\frac{f_{x}f_{y}(f_{zz}+f_{z})+f_{z}^{2}f_{xy}-f_{y}f_{z}f_{xz}-f_{x}f_{z}f_{yz}}{f_{z}^{3}}\right\vert_{{(1,1,1)}}.
  \end{align}
  Furthermore, if $v_{m}$ and $v_{n}$ are non-zero,
  then $M_{\ell}$ and $N_{\ell}$ are
  asymptotically normally distributed, respectively. If the
  variance-covariance matrix is non-singular, then $M_{\ell}$ and $N_{\ell}$ are asymptotically jointly normally distributed.
\end{lemma}
\begin{proof}The moment generating
  function is
  \begin{equation*}
    \mathbb E\exp(s_{1}M_{\ell}+s_{2}N_{\ell})=[z^{\ell}]e_{1}^{\top}(I-zA(e^{s_{1}},e^{s_{2}}))^{-1}w_{F}
  \end{equation*}
  where $e_{1}$ is a unit vector with a $1$ at the position of the initial
  state and the entries of $w_{F}$ are the exit weights of the states. Since the
  automaton is probabilistic and aperiodic, the unique dominant eigenvalue of
  $A(1,1)$ is $1$. Thus the same arguments apply as in
  \cite{Heuberger-Kropf-Wagner:2014:combin-charac} after replacing
  ``complete'' by ``probabilistic''. We obtain the same
  formulas for the constants of the expectation, the variance and the
  covariance. Also the central limit theorem follows. 
\end{proof}

\subsection{Standard Addition for
  $(q,d)$-expansions}\label{sec:standard-addition-q-d}
To construct the probabilistic automaton, we start with the transducer in
Figure~\ref{fig:stand-add-q-d}, and use the weights from
Section~\ref{sec:weights-q-d}.

All steps in this section, including the computation of the constants in
Theorem~\ref{thm:q-d-stand}, can be done in the mathematical software system
SageMath \cite{SageMath:2016:7.0} by using the included finite state
machine package described in \cite{Heuberger-Krenn-Kropf:ta:finit-state}. The
corresponding SageMath file is available at \cite{Heuberger-Kropf-Prodinger:2015:analy-carries-online}.

The construction in this section is more general than needed for the case of
independent digits as in $(q,d)$-expansions. But discussing it here in full
generality allows reusing the same ideas for the case of dependent digits as
in SSDEs
later on. We will use
the same construction for SSDEs in Sections~\ref{sec:stand-addit-ssde} and~\ref{sec:asympt-analys-von}.

In this section, let $\calA$ be the automaton in Figure~\ref{aut:q-d-exp}, equipped with the
weight $\frac 1q$ for every transition and the exit weight $1$ for every state (by
Section~\ref{sec:weights-q-d}). Construct $\calA^{2}$ as the \emph{additive
  Cartesian product}\footnote{This can also be seen as the composition of a
  transducer performing digitwise addition (without considering any carries)
  and the Cartesian product of $\calA$ with itself. This corresponds to the
  SageMath methods
  \href{http://www.sagemath.org/doc/reference/combinat/sage/combinat/finite_state_machine_generators.html\#sage.combinat.finite_state_machine_generators.TransducerGenerators.add}{\texttt{transducers.add}}
  and
  \href{http://www.sagemath.org/doc/reference/combinat/sage/combinat/finite_state_machine.html\#sage.combinat.finite_state_machine.Transducer.cartesian_product}{\texttt{Transducer.cartesian\_product}},
  respectively. The composition can be computed by the SageMath method \href{http://www.sagemath.org/doc/reference/combinat/sage/combinat/finite_state_machine.html\#sage.combinat.finite_state_machine.FiniteStateMachine.composition}{\texttt{Transducer.composition}}.} of $\calA$ with itself by the following rules:
\begin{itemize}
\item The states of $\calA^{2}$ are pairs of states of $\calA$.
\item There is a transition from $(a,b)$ to $(c,d)$ with label $x+y$ in $\calA^{2}$ if there
  are transitions from $a$ to $c$ with label $x$ and $b$ to $d$ with label $y$
  in $\calA$.
\item The weight of a transition in $\calA^{2}$ is the product of the weights
  of the two transitions in $\calA$.
\item The exit weight of a state in $\calA^{2}$ is the product of the exit
  weights of the two states in $\calA$.
\end{itemize}
The probabilistic automaton $\calA^{2}$ recognizes all possible sequences $\bfs$ of digitwise sums with the correct weights for the
equidistribution on the independent $(q,d)$-expansions $\bfx$ and
$\bfy$. 

In this section, let $\calB$ be the transducer in
Figure~\ref{fig:stand-add-q-d} performing the standard addition of two
$(q,d)$-expansions. Next, we construct $\calS_{(q,d)}$ as the
composition
$\calB\circ\calA^{2}$ by the following rules:
\begin{itemize}
\item The states of $\calS_{(q,d)}$ are pairs of states of $\calB$ and
  $\calA^{2}$.
\item For each pair of transitions from $a$ to $c$ with input label $s$ and
  output label $k$ in $\calB$ and from $b$ to $d$ with weight $w$ and
  label $s$ in $\calA^{2}$, there is a transition from $(a,b)$ to $(c,d)$ with
  weight $w$ and label $k$ in $\calS_{(q,d)}$.
\item The exit weight of a state in $\calS_{(q,d)}$ is the exit weight of the
  corresponding state in $\calA^{2}$.
\end{itemize}
The probabilistic automaton $\calS_{(q,d)}$ recognizes the sequence of carries
$\bfc$  with the correct weights for the
equidistribution on the independent $(q,d)$-expansions $\bfx$ and
$\bfy$. The probabilistic automaton $\calS_{(q,d)}$ has three states.

To determine the transition matrix of $\calS_{(q,d)}$, we use the following
lemma to compute the number of transitions between two states. The lemma
is proved by an inclusion-exclusion argument.
\begin{lemma}\label{lem:recursion-N}
  Let
  \begin{multline*}
N(x_{\min},x_{\max},y_{\min},y_{\max},s_{\min},s_{\max})=\\
\lvert\{(x,y)\in
\mathbb Z^{2}\mid x_{\min}\leq x\leq x_{\max}, y_{\min}\leq y\leq y_{\max},
s_{\min}\leq x+y\leq s_{\max}\}\rvert.
\end{multline*}
Then we have
{\allowdisplaybreaks
\begin{align*}
N(x_{\min},x_{\max},y_{\min},y_{\max},s_{\min},s_{\max})&=N(0,\infty,0,\infty,0,s_{\max}-x_{\min}-y_{\min}) \\
&\quad-N(0,\infty,0,\infty,0,s_{\max}-x_{\min}-y_{\max}-1)\\
&\quad-N(0,\infty,0,\infty,0,s_{\max}-x_{\max}-y_{\min}-1)\\
&\quad+N(0,\infty,0,\infty,0,s_{\max}-x_{\max}-y_{\max}-2)\\
&\quad-N(0,\infty,0,\infty,0,s_{\min}-x_{\min}-y_{\min}-1)\\
&\quad+N(0,\infty,0,\infty,0,s_{\min}-x_{\min}-y_{\max}-2)\\
&\quad+N(0,\infty,0,\infty,0,s_{\min}-x_{\max}-y_{\min}-2)\\
&\quad-N(0,\infty,0,\infty,0,s_{\min}-x_{\max}-y_{\max}-3)
\end{align*}}
with $N(0,\infty,0,\infty,0,s_{\max})=0$ if $s_{\max}$ is negative and
\begin{equation*}
  N(0,\infty,0,\infty,0,s_{\max})=\frac 12(s_{\max}+2)(s_{\max}+1)
\end{equation*}
otherwise.
\end{lemma}

This gives the transition matrix in
Table~\ref{tab:q-d-stand} in the appendix
 where $x$ marks carries $1$ and $y$ marks carries $-1$.  For example, the entry in
the first row and column is
\begin{align*}
  \frac{(d-1)(d-2)}{2q^{2}}y=\sum_{\substack{x,y\in D\\x+y\in
      L+1}}p_{0\rightarrow0}p_{0\rightarrow0}y=\frac 1{q^{2}}N(d,q+d-1,d,q+d-1,2d,d)y
\end{align*}
because this entry corresponds to the transitions from $-1$ to $-1$ with input
label $L+1$ and output label $\bar1$ in $\calB$ and from $(0,0)$ to $(0,0)$ in $\calA^{2}$.

With the transition matrix, the next theorem follows directly from Lemma~\ref{lem:asy-analy}.

\begin{theorem}\label{thm:q-d-stand}
  Let $M_{\ell}$ and $N_{\ell}$ be the
  number of carries $1$ and $-1$, respectively, when adding two independent random
  $(q,d)$-expansions of length $\ell$. The expected value of
  $(M_{\ell},N_{\ell})$ is $(e_{1},e_{-1})\ell+\bigOh(1)$ with constants
  \begin{align*}
    e_{1}&=\frac{(q + d - 1)^2}{2(q - 1)^2},\\
    e_{-1}&=\frac{d^{2}}{2(q-1)^{2}}.
  \end{align*}
  The variance-covariance matrix of $(M_{\ell},N_{\ell})$ is
  $\big(\begin{smallmatrix}v_{1}&c\\c&v_{-1}\end{smallmatrix}\big)\ell+\bigOh(1)$
  with constants
  \begin{align*}
    v_{1}&=\frac{(q + d - 1)^2  (q^4 - 2q^3d - q^2d^2 - 4qd^2 - 2q^2 - d^2 + 2d
+ 1)}{4(q - 1)^5  (q + 1)},\\
    v_{-1}&=\frac{d^{2}(2q^4 - q^2d^2 - 4q^3 - 6q^2d - 4qd^2 + 4q^2 + 6qd - d^2
      - 4q + 2)}{4(q - 1)^5  (q + 1)},\\
    c&=\frac{d(q+d-1)(q^3d + q^2d^2 - q^3 + 3q^2d + 4qd^2 + 2q^2 - 3qd + d^2 - q - d)}{4(q - 1)^5  (q + 1)}.
  \end{align*}
  Furthermore, the number of carries $1$ and $-1$ is asymptotically jointly normally
  distributed for $d\neq 0$, $-q+1$. For $d=0$, 
  $M_\ell$ is asymptotically normally distributed and $N_\ell=0$ because the
  carry $-1$ does not occur. For $d=-q+1$, the same holds with $M_\ell$ and
  $N_\ell$ exchanged.
\end{theorem}

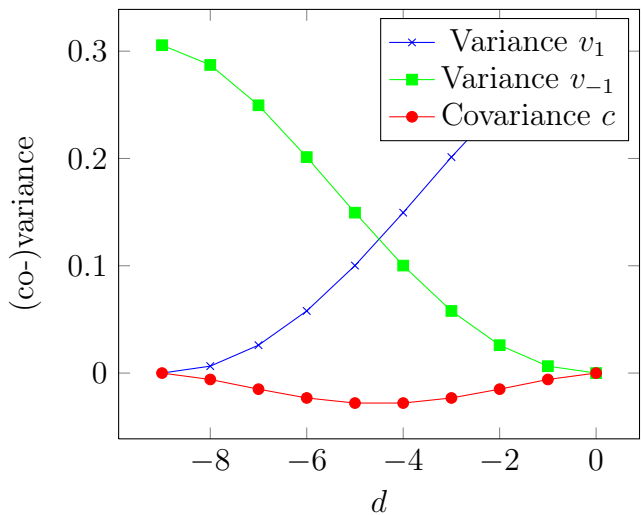
\begin{figure}
  \centering
  \begin{tikzpicture}
    \begin{axis}[
      xlabel=$d$,
      ylabel=(co-)variance]
    \addplot[color=blue,mark=x]
         table[x=d, y=varx] {table-var-standard-q-d};
    \addplot[color=green,mark=square*]
         table[x=d, y=vary] {table-var-standard-q-d};
    \addplot[color=red,mark=*]
         table[x=d, y=covar] {table-var-standard-q-d};
    \legend{Variance $v_{1}$, Variance $v_{-1}$,Covariance $c$}
    \end{axis}
  \end{tikzpicture}
  \caption{Variances and covariance for $(10,d)$-expansions of Theorem~\ref{thm:q-d-stand}.}
\end{figure}

\begin{remark}
  The expected value for carries in the addition of $(q,d)$-expansions corresponds to the result in \cite{Nakano-Sadahiro:2014:carries}. There,
  the authors find the
  stationary distribution \begin{equation*}\frac{1}{2(q-1)^{2}}(d^{2},\:
    q^{2}-2q+1-2qd+2d-2d^{2},\: (q+d-1)^{2})\end{equation*} for the states
  $(-1,0,1)$ of the carry process. For $d=\frac{-q+1}{2}$, this stationary distribution can
  also be found
  in \cite{Diaconis-Fulman:2014:combin}.
\end{remark}

\subsection{Standard Addition for SSDEs}\label{sec:stand-addit-ssde}
To cope with the dependencies between the digits, we have to combine the
conditional probabilities of the automaton in Figure~\ref{aut:ssde} with the
carries computed by the automaton in Figure~\ref{fig:stand-add-ssde}.
This is done in the same way as in Section~\ref{sec:standard-addition-q-d}.

All steps in this section, including the computation of the constants in
Theorem~\ref{thm:ssde-stand}, can be done in the mathematical software system
SageMath \cite{SageMath:2016:7.0} by using its included finite state
machine package described in \cite{Heuberger-Krenn-Kropf:ta:finit-state}. The
corresponding SageMath file is available at \cite{Heuberger-Kropf-Prodinger:2015:analy-carries-online}.

In this section, let $\calA$ be the automaton in Figure~\ref{aut:ssde}
equipped with the weights in \eqref{eq:trans-prob} and let $\calB$ be the
transducer in Figure~\ref{fig:stand-add-ssde} performing the standard addition
of two SSDEs.

We first construct the additive Cartesian product
$\calA^{2}$, recognizing all possible sequences $\bfs$ of
digitwise sums with the correct weights approximating the equidistribution on
two independent SSDEs $\bfx$ and $\bfy$. This
probabilistic automaton has $9$ states.

Next, we construct $\calS_{\SSDE}$ as the composition
$\calB\circ\calA^{2}$. This probabilistic automaton recognizes the sequence of
carries $\bfc$ with the correct weights approximating the
equidistribution on two independent SSDEs $\bfx$ and
$\bfy$. This gives a transducer with $45$ states.

Because of symmetries (cf.\ Remark~\ref{rem:interchange}), we can simplify
$\calS_{\SSDE}$ such that it has only $14$ states\footnote{For the actual
  computation, it is more efficient to already simplify $\calA^{2}$ by the
  SageMath method
  \href{http://www.sagemath.org/doc/reference/combinat/sage/combinat/finite_state_machine.html\#sage.combinat.finite_state_machine.FiniteStateMachine.markov_chain_simplification}{\texttt{FiniteStateMachine.markov\_chain\_simplification}},
  such that
  it only has $6$ states.}:
\begin{lemma}\label{lem:simplification} A probabilistic automaton can be simplified by applying the following rules:
  \begin{itemize}
  \item If between two states, there are two transitions with the same label,
    then these two transitions can be combined. The weights are summed up in
    this process.
  \item Let $\{C_{1},\ldots,C_{k}\}$ be a partition of the states of the
    automaton with the following property:
    If $a$, $b\in C_{j}$ are two states, then there is a bijection between the
      transitions leaving $a$ and the ones leaving $b$ which preserves the
      label, the weight of the transition and into which set of the partition the transitions
      lead. These bijections define an equivalence relation on the transitions
      leaving a set of the partition.

Then each set of the partition can be contracted to a new state. For each
equivalence class of transitions,
there is one transition in the simplified transducer. 
  \end{itemize}
\end{lemma}

Thus, we obtain a $14\times14$ transition matrix of $\calS_{\SSDE}$ given in
Table~\ref{tab:ssde-stand} in the
appendix (using Lemma~\ref{lem:recursion-N}).

\begin{theorem}\label{thm:ssde-stand}
  The expected value of the number of carries equal to $1$ when adding two
  SSDEs of length $\ell$ is
  \begin{equation*}
    \frac {q^{2} + 2 q + 4}{8(q + 1)^{2}} \ell+\bigOh(1)
  \end{equation*}
  and the variance is
  \begin{equation*}
    \frac {7
q^{6} + 48 q^{5} + 159 q^{4} + 128 q^{3} - 48 q^{2} - 12 q - 8}{64(q + 1)^{5} (q - 1)}\ell+\bigOh(1).
  \end{equation*}

The same result holds for carries equal to $-1$. The covariance between
carries $1$ and $-1$ is
\begin{equation*}
  -\frac {q^6 + 24q^5 + 33q^4 + 80q^3 + 120q^2 - 12q - 8}{64(q + 1)^{5} (q - 1)}\ell+\bigOh(1).
\end{equation*}

The number of carries $1$ and $-1$ is asymptotically jointly normally distributed.
\end{theorem}

\begin{figure}
  \centering
  \begin{tikzpicture}
    \begin{axis}[
      xlabel=$q$,
      ylabel=(co-)variance]
    \addplot[color=blue]
         table[x=q, y=varx] {table-var-standard-ssde};
    \addplot[color=red]
         table[x=q, y=covar] {table-var-standard-ssde};
    \legend{Variance,Covariance}
    \end{axis}
  \end{tikzpicture}
  \caption{Variance and covariance for SSDEs for $q=2,\ldots,100$ of Theorem~\ref{thm:ssde-stand}.}
  \label{fig:ssde-var-covar}
\end{figure}
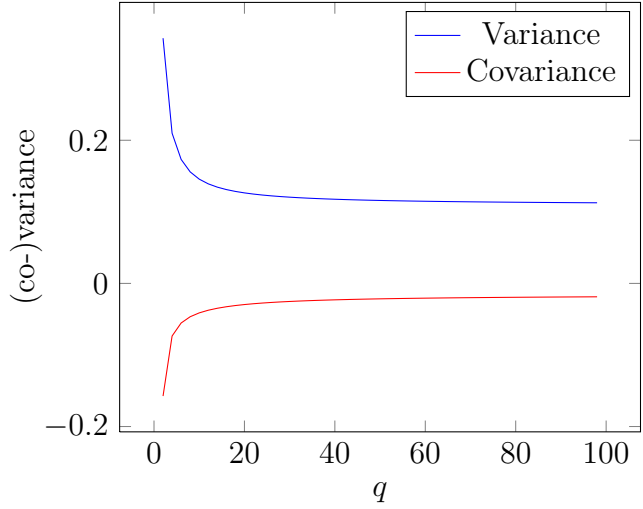
\begin{proof}
  We can compute the determinant $f(x,y,z)=\det(I-zA(x,y))$ of the transition
  matrix $A(x,y)$ in the appendix of the simplified
  automaton $\calS_{\SSDE}$ with $14$ states.
Thus, Lemma~\ref{lem:asy-analy} implies the expected value, the variance and the
central limit theorem where the input sequence is the sum of two independent SSDEs of
length $\ell$ with the approximate equidistribution $W_{\ell}$. 

As the (exact) equidistribution $\mathbb P_{\ell}$ satisfies $\mathbb
P_{\ell}=(1+\bigOh(\xi^{\ell}))W_{\ell}$, these results also hold for the (exact) equidistribution.
\end{proof}

\begin{remark}\label{rem:ignore-dep}
  If we neglect the dependencies between two adjacent digits, we obtain a
  different result: Assume that the digits are independently distributed with
  probabilities given in \eqref{eq:weights-digits-ssde}. Then the expected
  value of the number of carries $1$ is
  \begin{equation*}
    \frac{q^{12} + 7 q^{11} + 19 q^{10} + 27 q^{9} + 24 q^{8} + 9 q^{7} - 15 q^{6} - 15 q^{5} + 47 q^{4} + 104 q^{3} + 64 q^{2} - 48 q - 48}{8(q + 1)^{3} q^{2} (q^{7} + 4 q^{6} + 5 q^{5} -  q^{4} - 9 q^{3} - 8 q^{2} + 4)}
  \end{equation*}
and the variance is
\begin{align*}
  &\frac{1}{64}(7 q^{38} + 152 q^{37} + 1557 q^{36} + 9958 q^{35} + 44300 q^{34} + 144166
  q^{33} + 349511 q^{32}\\
 &\qquad+ 622942 q^{31} + 756995 q^{30} + 432788 q^{29} - 439628 q^{28} - 1347486
 q^{27} - 1407649 q^{26}  \\
&\qquad- 466340 q^{25}- 39181 q^{24} - 2293904 q^{23} - 6902413 q^{22} - 9055044 q^{21} - 2972395
q^{20} \\
 &\qquad+ 10157788 q^{19} + 19040707 q^{18}+ 12034998 q^{17} - 7655356 q^{16} - 21471482 q^{15}  \\
&\qquad - 15688011 q^{14}+
 1495584 q^{13} + 10611092 q^{12} + 5762536 q^{11}- 1482784 q^{10} \\ 
&\qquad - 1794016 q^{9} + 1000784 q^{8} + 744768 q^{7}-
1199872 q^{6} - 1204224 q^{5} + 120832 q^{4} \\
&\qquad + 574464 q^{3}+ 172032 q^{2} - 73728 q - 36864)\\
& \times q^{-4}(q + 1)^{-6}  (q^{7} + 4 q^{6} +
5 q^{5} -  q^{4} - 9 q^{3} - 8 q^{2} + 4)^{-3}\\
&\times
 (q^{7} + 2 q^{6} + q^{5} + q^{4} + q^{3} - 2 q^{2} + 4)^{-1}.
\end{align*}
As expected, the limit for $q$ to infinity is the same.
\end{remark}
\section{Von Neumann's Addition}\label{sec:von-neum-addit}

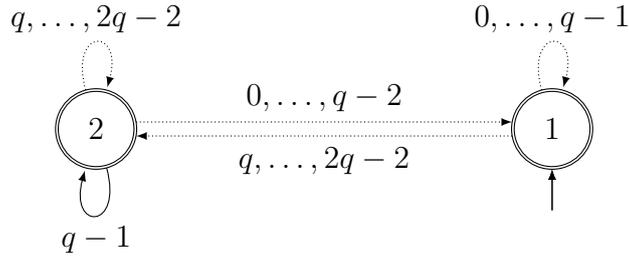
\begin{figure}
  \centering
  \begin{tikzpicture}[auto, initial text=, >=latex]
    \node[state, accepting, initial, initial below] (v0) at (6.000000, 0.000000) {$1$};
    \node[state, accepting] (v1) at (0.000000, 0.000000) {$2$};
    \path[->] (v0.190.00) edge[densely dotted] node[rotate=360.00, anchor=north] {$q,\ldots,2q-2$} (v1.350.00);
    \path[->] (v1.10.00) edge[densely dotted] node[rotate=0.00, anchor=south] {$0,\ldots,q-2$} (v0.170.00);
    \path[->] (v0) edge[loop above, densely dotted] node {$0,\ldots,q-1$} ();
    \path[->] (v1) edge[loop above, densely dotted] node {$q,\ldots, 2q-2$} ();
    \path[->] (v1) edge[loop below] node {$q-1$} ();
  \end{tikzpicture}
  \caption{Automaton to find the longest carry generating sequence for von
    Neumann's addition of two standard $q$-ary expansions.}
  \label{fig:neumann-q-0}
\end{figure}

In this section, we analyze von Neumann's addition algorithm for SSDEs, a parallel
algorithm using several iterations. This
algorithm was analyzed by Knuth in \cite{Knuth:1978:carry} for standard $q$-ary
expansions. In \cite{Heuberger-Prodinger:2003:carry-propag}, this analysis was
extended to $(q,d)$-expansions and SSDEs. However, for $q\geq 4$, the hardware
and software available at that time made the use of the probabilistic model of
Section~\ref{sec:weights-ssde} computationally infeasible. The
approximate model described in Remark~\ref{rem:ignore-dep} was used
instead. As Remark~\ref{rem:ignore-dep} demonstrates, this approximation may
lead to different main terms in the expectation and the variance.

As before, we choose an approximate equidistribution for all independent pairs of SSDEs of length $\ell$
as our probabilistic input model. In contrast to
the result in \cite{Heuberger-Prodinger:2003:carry-propag}, we
obtain more
natural constants occurring in the main term of the expectation and the variance.

For von Neumann's addition of two standard $q$-ary digit expansions, the number of
iterations depends on the longest subsequence $(q-1)\ldots(q-1)j$ with $j\geq
q$ of the digitwise sum $\bfs$, see \cite{Knuth:1978:carry}. Such sequences can be found by an
automaton with two classes of transitions (see Figure~\ref{fig:neumann-q-0}
and \cite[Figure~1]{Heuberger-Prodinger:2003:carry-propag}). One class
corresponds to the digit $(q-1)$ of a carry generating
sequence and is depicted by solid lines. The other class corresponds to all
other digits (including the digit $j$ of a carry generating sequence) and is
depicted by dotted lines. The longest
consecutive run of solid edges in the automaton in
Figure~\ref{fig:neumann-q-0} corresponds to the number of iterations of von Neumann's addition minus
$2$. The asymptotic analysis of these longest runs can be performed using the
probabilistic version of the automaton in Figure~\ref{fig:neumann-q-0}. We will extend
this approach to SSDEs with arbitrary even base using a larger probabilistic automaton in Section~\ref{sec:asympt-analys-von}.

\subsection{Algorithm}
\begin{table}
  \begin{center}      \def\1{\bar 1} \def\2{\bar 2} \setlength{\extrarowheight}{2pt}
$$    \begin{array}{>{(}r<{)_{4}}@{\;=\;}>{}r<{}@{\;=\;}r}
      110\1\2& \bfx=\bfz^{(0)}&314\\
      101\1\2&  \bfy=\bfc^{(0)}&266\\\hline
      21120& \bfz^{(1)}&600\\
      000\1\10& \bfc^{(1)}&-20\\\hline
      021010& \bfz^{(2)}&580\\
      0000000&\bfc^{(2)}&0
      \end{array}$$
    \caption{Example for von Neumann's addition for SSDEs with $q=4$. We have $t(110\bar1\bar2,101\bar1\bar2)=2$.}
    \label{tab:intro-ssde}
  \end{center}
\end{table}
Let $\bfx$ and $\bfy$ be two SSDEs.
The idea of the algorithm is to construct the sequence of digitwise sums
$\bfs=\bfx+\bfy$ and correct each position if the
number at this position is not in the digit set or at the border of the digit
set where we have to take into account the syntactical rule. 

As in \cite{Heuberger-Prodinger:2003:carry-propag}, we define
$(\bfz,\bfc)=\add(\bfs)$ with $\bfs=\bfx+\bfy$ by
\begin{align*}
  c_{0}&=0,\\
  c_{j+1}&=
  \begin{cases}
    \sgn(s_{j})&\text{if }|s_{j}|>\frac q2,\\
    &\quad\text{or }|s_{j}|=\frac q2
    \text{ and } \\
    &\quad(\sgn(s_{j})s_{j+1})\bmod q\geq \frac q2\\
    0&\text{otherwise,}
  \end{cases}\\
  z_{j}&=s_{j}-c_{j+1}q.
\end{align*}
Here, the choice of the carry $c_{j+1}$ corresponds to the one in
Algorithm~\ref{alg:stand-add-ssde}.
By iterating this step we obtain
$(\bfz^{(k+1)},\bfc^{(k+1)})=\add(\bfz^{(k)}+\bfc^{(k)})$ with
$\bfz^{(0)}=\bfx$ and $\bfc^{(0)}=\bfy$. If $\bfc^{(k)}=0$,
then $\bfz^{(k)}$ is the SSDE of the sum $\bfx+\bfy$ and the algorithm
stops. Note that during this process, $\bfz^{(k)}$ and $\bfc^{(k)}$ are not
necessarily SSDEs.

 In
\cite{Heuberger-Prodinger:2003:carry-propag}, the correctness and the termination of this
algorithm were proved. We denote the number of iterations of von Neumann's
addition algorithm by $t(\bfx,\bfy)=\min\{k\geq 0: \bfc^{(k)}=0\}$.

\subsection{Automaton}
\begin{figure}
  \centering
  \begin{tikzpicture}[auto, initial text=, >=latex, scale=0.88]
\tikzstyle{every initial by arrow}=[densely dotted]
\node[state, accepting, initial, initial below] (v0) at (0.000000, 0.000000) {$1$};
\node[state, accepting] (v1) at (7.000000, 7.000000) {$2$};
\node[state, accepting] (v2) at (7.000000, -7.000000) {$3$};
\node[state, accepting] (v3) at (4.000000, 0.000000) {$4$};
\node[state, accepting] (v4) at (5.80000, -2.000000) {$5$};
\node[state, accepting] (v5) at (-7.000000, 7.000000) {$7$};
\node[state, accepting] (v6) at (-7.000000, -7.000000) {$8$};
\node[state, accepting] (v7) at (-4.000000, 0.000000) {$9$};
\node[state, accepting] (v8) at (-5.80000, -2.000000) {$10$};
\path[use as bounding box] (-8.7,-8.7) rectangle (8.7,8.7);

\path[->] (v0) edge[in=-110, out=-70, densely dotted, looseness=1000, loop] 
    node[below] {$L\cup -L$} 
    ();
\path[->] (v0.5.00) edge[densely dotted] 
    node[sloped, anchor=south]{$H$} 
    (v3.175.00);
\path[->] (v0.-10.95) edge[densely dotted] 
    node[anchor=south, sloped, pos=0.65]{$\frac{q}{2}$} 
    (v4.159.05);
\path[->] (v0.185.00) edge[densely dotted] 
    node[sloped, anchor=north, pos=0.4]{$-H$} 
    (v7.355.00);
\path[->] (v0.-159.05) edge[densely dotted] 
    node[sloped, anchor=north]{$-\frac{q}{2}$} 
    (v8.10.95);

\path[->] (v1.-140.00) edge[densely dotted] 
    node[sloped, above] {$-L_{0}$} 
    (v0.50.00);
\path[->] (v1.-130.00) edge 
    node[sloped, below] {$L_{0}-1$} 
    (v0.40.00);
\path[->] (v1) edge[loop above, densely dotted] 
    node {$\frac{q}{2}$} 
    ();
\path[->] (v1.-85.00) edge 
    node[sloped, above]{$\frac{q}{2} - 1$} 
    (v2.85.00);
\path[->] (v1.-118.20) edge[densely dotted] 
    node[sloped, anchor=north]{$H_{q}$} 
    (v3.71.80);
\path[->] (v1.-108.20) edge[bend left=10] 
    node[sloped, anchor=north, style= near end]{$q$} 
    (v3.61.80);
\path[->] (v1.185.00) edge 
    node[sloped, anchor=north]{$-\frac{q}{2} - 1$} 
    (v5.355.00);
\draw[->, densely dotted] (v1) .. controls (10,-10) .. (v6)
     node[sloped, below, style=near start] {$-\frac{q}{2}$};
\path[->] (v1.-152.53) edge 
    node[sloped, anchor=north] {$-H-1$}
    (v7.27.47);

\path[->] (v2.130.00) edge 
    node[sloped, anchor=south] {$ -L_{0}$} 
    (v0.-40.00);
\path[->] (v2.140.00) edge[densely dotted] 
    node[sloped, below] {$L$} 
    (v0.-50.00);
\path[->] (v2.95.00) edge 
    node[sloped, anchor=south]{$\frac{q}{2}$} 
    (v1.265.00);
\path[->] (v2.118.20) edge[densely dotted] 
    node[sloped, anchor=north]{$ q$} 
    (v3.288.20);
\path[->] (v2.123.20) edge[bend left=10] 
    node[sloped, anchor=north, pos=0.75]{$H_{q}$}
    (v3.278.20);
\path[->] (v2.185.00) edge 
    node[sloped, anchor=north]{$-\frac{q}{2}$} 
    (v6.355.00);
\path[->] (v2.152.53) edge[densely dotted] 
    node[sloped, anchor=north]{$-H$} 
    (v7.322.53);

\path[->] (v3.185.00) edge[densely dotted] 
    node[sloped, anchor=north] {$-L_{0}\cup(L_{0}-1)$} 
    (v0.355.00);
\path[->] (v3.81.80) edge[densely dotted] 
    node[sloped, anchor=south]{$\frac{q}{2}$} 
    (v1.231.80);
\path[->] (v3.-61.80) edge[densely dotted] 
    node[sloped, anchor=south]{$\frac{q}{2} - 1$} 
    (v2.108.20);
\path[->] (v3) edge[loop right, densely dotted] 
    node {$H$} 
    ();
\path[->] (v3.142.53) edge[densely dotted] 
    node[sloped, anchor=north]{$-\frac{q}{2}- 1$} 
    (v5.332.53);
\path[->] (v3.-142.53) edge[densely dotted] 
    node[sloped, anchor=north]{$-\frac{q}{2}$} 
    (v6.27.47);
\draw[->, densely dotted] (v3) .. controls (0,1.8)  .. (v7) 
    node[sloped,above, style=midway] {$-H-1$};


\path[->] (v4.169.05) edge[densely dotted] 
    node[sloped, anchor=north] {$-L\cup L$} 
    (v0.339.05);
\path[->] (v4) edge[densely dotted] 
    node[sloped, anchor=south] {$\frac{q}{2}$}  
    (v1.-100.00);
\path[->] (v4) edge[densely dotted] 
    node[sloped, anchor=south] {$H$} 
    (v3);
\path[->] (v4) edge[densely dotted] 
    node[sloped, below] {$-\frac{q}{2}$} 
    (v6);
\path[->] (v4) edge[densely dotted, bend left=10] 
    node[sloped, below, pos=0.25]{$-H$} 
    (v7);

\path[->] (v5.-40.00) edge[densely dotted] 
    node[sloped, anchor=south] {$ L_{0}$} 
    (v0.130.00);
\path[->] (v5.-50.00) edge 
    node[sloped, below] {$-L_{0}+1$} 
    (v0.140.00);
\path[->] (v5.5.00) edge 
    node[sloped, anchor=south]{$\frac{q}{2}+ 1$} 
    (v1.175.00);
\draw[->, densely dotted] (v5) .. controls (-10,-10) .. (v2)
     node[sloped, anchor=north, style=near start] {$\frac{q}{2}$};
\path[->] (v5.-17.47) edge 
    node[sloped, anchor=south] {$H+1$} 
    (v3.132.53);
\path[->] (v5) edge[loop above, densely dotted]  
    node {$-\frac{q}{2}$} 
    ();
\path[->] (v5.-85.00) edge 
    node[sloped, anchor=south]{$-\frac{q}{2}+ 1$} 
    (v6.85.00);
\path[->] (v5.-51.80) edge[densely dotted]  
    node[sloped, anchor=south]{$-H_{q}$} 
    (v7.98.20);
\path[->] (v5.-71.80) edge[bend right=10] 
    node[sloped, below, style= near end]{$ -q$} 
    (v7.118.20);

\path[->] (v6.40.00) edge[densely dotted] 
    node[sloped, below] {$-L$} 
    (v0.-130.00);
\path[->] (v6.50.00) edge 
    node[sloped, anchor=south] {$ L_{0}$} 
    (v0.-140.00);
\path[->] (v6.5.00) edge 
    node[sloped, anchor=south]{$\frac{q}{2}$} 
    (v2.175.00);
\path[->] (v6.37.47) edge[densely dotted] 
    node[sloped, anchor=south]{$H$} 
    (v3.207.47);
\path[->] (v6.95.00) edge 
    node[sloped, anchor=south]{$-\frac{q}{2}$} 
    (v5.265.00);
\path[->] (v6.56.80) edge[bend right=10] 
    node[sloped, below, pos=0.75]{$-H_{q}$} 
    (v7.-98.20);
\path[->] (v6.71.80) edge[densely dotted] 
    node[sloped, anchor=south]{$ -q$} 
    (v7.241.80);

\path[->] (v7.5.00) edge[densely dotted] 
    node[sloped, anchor=south] {$(-L_{0}+1)\cup L_{0}$} 
    (v0.175.00);
\path[->] (v7.37.47) edge[densely dotted] 
    node[sloped, anchor=south]{$\frac{q}{2} + 1$} 
    (v1.197.47);
\path[->] (v7.-27.47) edge[densely dotted] 
    node[sloped, anchor=south]{$\frac{q}{2}$} 
    (v2.142.53);
\path[->] (v7.20.00) edge[densely dotted, bend left=20] 
    node[sloped, anchor=south] {$H+1$} 
    (v3.160.00);
\path[->] (v7.108.20) edge[densely dotted] 
    node[sloped, anchor=north]{$-\frac{q}{2}$} 
    (v5.298.20);
\path[->] (v7.-108.20) edge[densely dotted] 
    node[sloped, anchor=north]{$-\frac{q}{2} + 1$} 
    (v6.61.80);
\path[->] (v7) edge[loop left, densely dotted] 
    node {$-H$} 
    ();

\path[->] (v8.20.95) edge[densely dotted] 
    node[sloped, anchor=south] {$L\cup -L$} 
    (v0.190.95);
\path[->] (v8) edge[densely dotted] 
    node[sloped, below] {$\frac{q}{2}$} 
    (v2);
\path[->] (v8) edge[densely dotted, bend right=10] 
    node[sloped, below, pos =0.25] {$H$} 
    (v3);
\path[->] (v8) edge[densely dotted] 
    node[sloped, anchor=south, pos=0.45] {$-\frac{q}{2}$} 
    (v5.-80.00);
\path[->] (v8) edge[densely dotted] 
    node[sloped, anchor=south]{$-H$} 
    (v7);
\end{tikzpicture}
  \caption{Automaton in \cite[Figure~5]{Heuberger-Prodinger:2003:carry-propag}: $t(\bfx,\bfy)\leq k+2$ if and only if the automaton traverses at most $k$ solid edges when reading $(s_j)_{j\geq 0}$.}
  \label{aut:ssde-von-neu}
\end{figure}
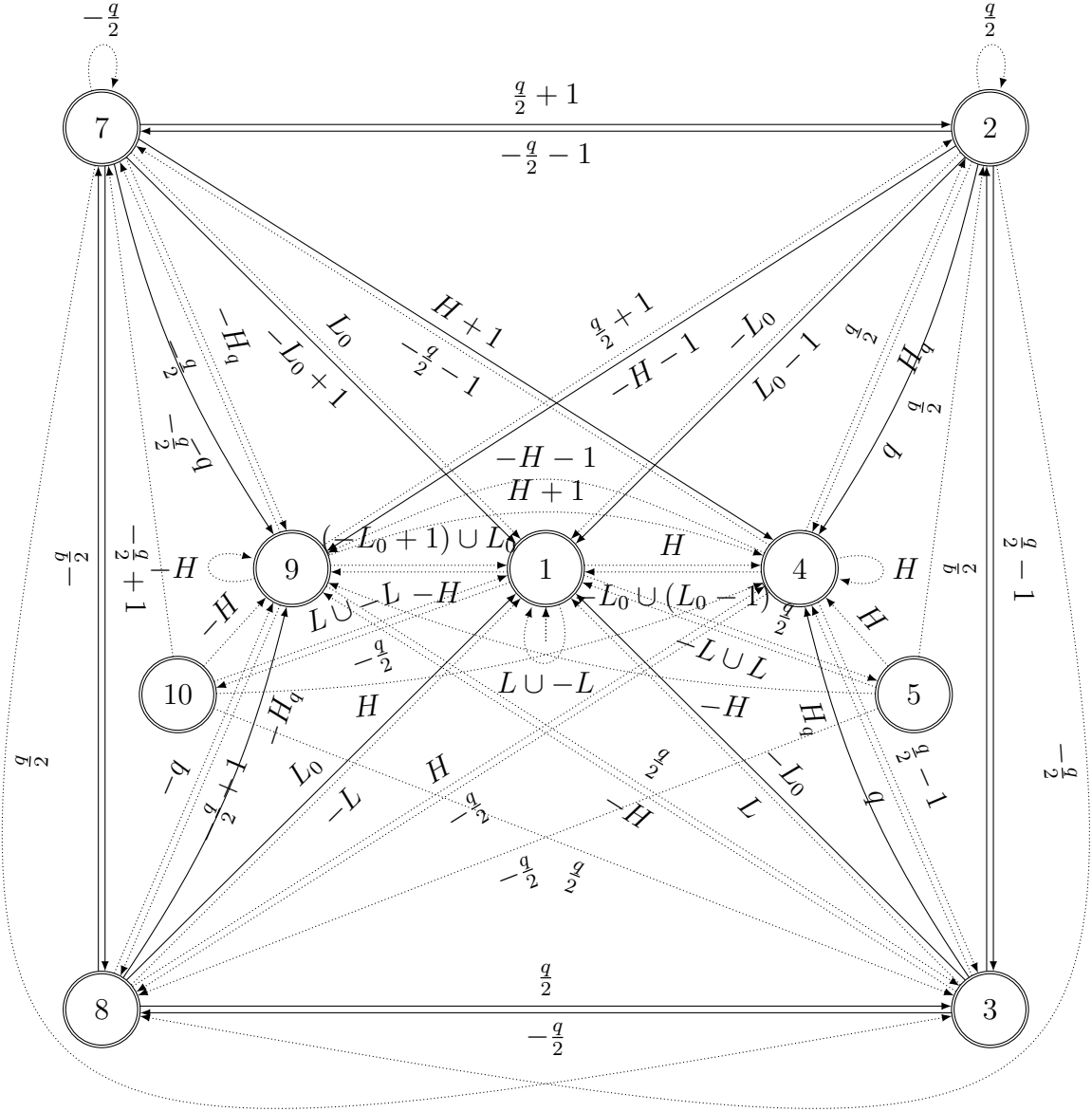
A description of all SSDEs $\bfx$ and $\bfy$ with $t(\bfx,\bfy)=k$ is
given in \cite{Heuberger-Prodinger:2003:carry-propag}. This description is in terms of an automaton  and leads to the
automaton in \cite[Figure~5]{Heuberger-Prodinger:2003:carry-propag} reproduced
here as Figure~\ref{aut:ssde-von-neu}. We use the sets $L=\{0,\ldots,q/2-1\}$,
$L_{0}=L\setminus\{0\}$, $H=\{q/2+1,\ldots, q\}$ and $H_{q}=H\setminus\{q\}$.

From
\cite[Theorem~3.4]{Heuberger-Prodinger:2003:carry-propag}, we know that
$t(\bfx,\bfy) \leq k+2$ if and only if this automaton traverses at most
$k$ consecutive solid transitions when reading $\bfs$.
\begin{remark}\label{rem:read-inf-sequ}
Strictly speaking, the automaton reads the sequence
$(s_{j})_{j\geq 0}$ where $s_{j}=x_{j}+y_{j}$ for $j\leq J$ and $s_{j}=0$ for
$j>J$, for some $J$. However, most of the solid edges are visited while $j\leq
J$. All transitions with label $0$ lead to state $1$. Those from states $2$
and $7$ are solid edges, all others are dotted. If the transition is in state
$2$ (or $7$) after reading $s_{J}$, an additional solid edge will be traversed. Thus, we have to specially
treat the states
$2$ and $7$.
\end{remark}

\section{Asymptotic Analysis of von Neumann's Addition}\label{sec:asympt-analys-von}
For the asymptotic analysis, we combine the automaton in
Figure~\ref{aut:ssde-von-neu} with the probabilistic model for SSDEs from
Section~\ref{sec:weights-ssde} in the same way as in
Section~\ref{sec:stand-addit-ssde}.

All steps in this section, including the computation of the constants in
Theorem~\ref{thm:ssde-neumann}, can be done in the mathematical software system
SageMath \cite{SageMath:2016:7.0} by using the included finite state
machine package described in \cite{Heuberger-Krenn-Kropf:ta:finit-state}. The
corresponding SageMath file is available at \cite{Heuberger-Kropf-Prodinger:2015:analy-carries-online}.

We again use the automata $\mathcal A$ and $\mathcal
A^{2}$ described in Section~\ref{sec:stand-addit-ssde}, recognizing SSDEs and
the digitwise sum of two SSDEs, respectively. As before, the next step is to construct the
Cartesian product $\calN_{\SSDE}$ of the automaton $\calB$ in
Figure~\ref{aut:ssde-von-neu} and  $\mathcal A^{2}$.

After simplifying this construction as described in Lemma~\ref{lem:simplification}, the probabilistic automaton $\calN_{\SSDE}$
has $12$ states:
{\allowdisplaybreaks
\begin{equationaligned}
&\{(1,(-1, 1)),(1,(1, -1))\},\OBHnotag\\
&\{(4,(0, 0)), (9,(0, 0))\},\OBHnotag\\
&\{(5,(0, 1)), (5,(1, 0)), (10,(-1,0)), (10,(0, -1))\},\label{eq:states-ord}\\
&\{(2,(0, 1)), (2,(1, 0)), (7,(-1, 0)), (7,(0, -1))\},\OBHnotag\\
&\{(5,(0, 0)),(10,(0, 0))\},\OBHnotag\\
&\{(2,(0,0)), (7,(0, 0))\},\OBHnotag\\
&\{(1,(-1, 0)), (1,(0, -1)), (1,(0, 1)), (1,(1, 0))\},\OBHnotag\\
&\{(3,(0, 1)), (3,(1, 0)), (8,(-1, 0)),(8,(0, -1))\},\OBHnotag\\
&\{(1,(0, 0))\},\OBHnotag\\
&\{(3,(0, 0)), (8,(0, 0))\},\OBHnotag\\
&\{(4,(1, 1)), (9,(-1, -1))\},\OBHnotag\\
&\{(4,(0, 1)), (4,(1, 0)), (9,(-1, 0)),(9,(0, -1))\}.\OBHnotag
\end{equationaligned}}
In this case, the simplification is done in the same way as in
Lemma~\ref{lem:simplification}, but also taking into account the class (dotted
or solid) of a transition. The partition of the set of states was constructed
by the symmetries between the two sequences $\bfx$ and $\bfy$ described in
Remark~\ref{rem:interchange}, for example $\{(1,(-1, 1)),(1,(1, -1))\}$, and the additional vertical symmetry of the
automaton in Figure~\ref{aut:ssde-von-neu}, for example $\{(4,1, 1),(9,-1, -1)\}$.

The state $(1,(0,0))$ is initial and all states are final.

The next theorem is an extension of Lemma~2.5 in
\cite{Heuberger-Prodinger:2003:carry-propag} additionally including the variance and convergence
in distribution.
\begin{theorem}\label{theorem:bootstrapping}
   Let $w_{\ell k}$, $\ell$, $k \ge 0$, be non-negative numbers
  with generating function
  \begin{equation*}
    G_{k}(z)=\frac{R_k(z)}{S_k(z)}=\sum_{\ell \ge 0} w_{\ell k} z^\ell
  \end{equation*}
such that $w_{\ell k}$ is non-decreasing in $k$.
  
  Assume that
  \begin{align*}
    R_k(z)&=r_0(z)+ r_{1}\Big(z,\Big(\frac
    z{a_{1}}\Big)^{k},\ldots,\Big(\frac z{a_{m}}\Big)^{k}\Big),\\
    S_k(z)&=(1-z)s_0(z)+\Big(\frac z{a_1}\Big)^k s_1(z)+
    s_{2}\Big(z,\Big(\frac z{a_{1}}\Big)^{k},\ldots,\Big(\frac z{a_{m}}\Big)^{k}\Big),
  \end{align*}
  where $r_0$, $s_0$, and $s_1$ are real polynomials in $z$
  (not depending on $k$). Furthermore,
  $r_{1}$ and $s_{2}$ are real polynomials in $z$, $(z/a_1)^k$, \dots, $(z/a_m)^k$
  for some $m\ge 2$ and some real numbers $1<a:=a_1<\lvert a_2\rvert\le\lvert
a_3\rvert\le
  \dots\le \lvert a_m\rvert$ such that each of the summands in $r_{1}$ is
  divisible by one of the terms
  $(z/a_{1})^{k}$, \ldots, $(z/a_{m})^{k}$ and each of the summands in $s_{2}$
  is  divisible by one of the terms
  $(z/a_{1})^{2k}$, $(z/a_{2})^{k}$, \ldots, $(z/a_{m})^{k}$. Define
  \begin{equation*}
    \delta:=s_1(1)/s_0(1),\qquad
    \rho:=\min\left(\log\lvert a_2\rvert/\log a_1, 2\right)-1.
  \end{equation*}
  Assume
  furthermore that $r_{0}(1)\neq 0$, that $s_0$ does not have any zero in
  $\lvert z\rvert\le 1$ and that $\delta>0$.

  Then $G(z):=\frac{r_{0}(z)}{(1-z)s_{0}(z)}=\lim_{k\to\infty} G_k(z)$ and
  \begin{equation*}
    G(z)=\sum_{\ell \geq 0}w_{\ell }z^{\ell }
  \end{equation*}
  with $w_{\ell}=w_{\ell k}$ for $k\geq \ell$. Additionally, $w_{\ell}\neq 0$
  for $\ell\geq\ell_{0}$ for a suitable $\ell_{0}$.

  Let $(X_{\ell})_{\ell\geq\ell_{0}}$ be the sequence of random variables with support $\mathbb N_{0}$
  defined by
  \begin{equation*}
    \mathbb P(X_{\ell }\leq k)=\frac{w_{\ell k}}{w_{\ell }}.
  \end{equation*}
  
  Then the asymptotic formula
  \begin{equation}\label{eq:2}
    \frac{w_{\ell k}}{w_{\ell }}=\exp(-\delta\ell/a^{k})(1+o(1))
  \end{equation}
  holds as $\ell \rightarrow \infty$ for $k=\log_{a} \ell +\bigOh(1)$. Hence the
  shifted random variable $X_{\ell }-\log_{a}\ell $ converges weakly to a limiting
  distribution if $\ell $ runs through a subset of the positive integers such that
  the fractional part $\{\log_{a}\ell \}$ of $\log_{a}\ell $ converges.

The expected value of $X_{\ell }$ is 
  \begin{equation}\label{eq:3}
    \mathbb EX_{\ell}=\log_a\ell + \log_a\delta +\frac{\gamma}{\log
      a}+\frac12+
    \Psi_{0}(\log_a \ell+\log_a\delta)
    + \bigOh\left(\frac{\log^{\rho+3} \ell}{\ell^\rho}\right),
  \end{equation}
  the variance is
  \begin{multline}\label{eq:6}
    \mathbb
    VX_{\ell}=\frac{\pi^{2}}{6\log^{2}a}+\frac{1}{12}+\Psi_{1}(\log_{a}\ell+\log_{a}\delta)-\frac{2\gamma}{\log
    a}\Psi_{0}(\log_{a}\ell+\log_{a}\delta)\\-\Psi_{0}^{2}(\log_{a}\ell+\log_{a}\delta)+\bigOh\Big(\frac{\log^{\rho+4}\ell}{\ell^{\rho}}\Big),
  \end{multline}
  where $\gamma$ is the Euler--Mascheroni constant, and $\Psi_{0}(x)$ and $\Psi_{1}(x)$ are periodic functions (with period $1$ and mean value $0$), 
  given by the Fourier expansions
  \begin{align}\label{eq:le:asymptotik:psi}
    \Psi_{0}(x)=-\frac1{\log 
      a}\sum_{n\neq0}\Gamma\Big(-\frac{2n\pi 
      i}{\log a}
    \Big)e^{2n\pi ix},\\
    \Psi_{1}(x)=\frac{2}{\log^{2}a}\sum_{n\neq0}\Gamma'\Big(-\frac{2n\pi 
      i}{\log a}
    \Big)e^{2n\pi ix}.
  \end{align}
\end{theorem}
\begin{proof}
Parts of the proof of this theorem follow along the same lines as the proof of
Lemma~2.5 in~\cite{Heuberger-Prodinger:2003:carry-propag}. However, we
include all steps of the proof for the
sake of readability.

Without loss of generality, we can assume $r_{0}(1)/s_{0}(1)=1$,
  as otherwise $w_{\ell k}$ and $w_{\ell}$ are multiplied by a
  constant. Let $0\le k_1\le k_2\le k_3$ denote suitable constants.

For some $C>0$ such that there is no root of $s_0$ inside $\{z:\lvert z\rvert\le
  1+2C\}$ and such that $(1+C)/a<1$, we have
  \begin{equation*}
    \lvert S_k(z)-(1-z)s_0(z)\rvert=\bigOh\Bigl(\bigl((1+C)/a\bigr)^k\Bigr)<\lvert(1-z)s_0(z)\rvert 
  \end{equation*}
  for $\lvert z\rvert=1+C$ and $k\ge k_1$. By Rouch\'e's Theorem, we conclude that
  for $k\ge k_1$,
  $S_k(z)$ has exactly one simple root in the disk $\{z:\lvert z\rvert\le
  1+C\}$.

Since $\sgn(S_k(1))=\sgn(s_1(1))$ and $\sgn(S_k(1+1/k))=-\sgn(s_0(1))$ for
  $k\ge k_2$, the assumption $\delta>0$ implies that $S_k(z)$ has a real root
  $\zeta_k=1+\varepsilon_k$ with $0<\varepsilon_k<1/k$ for $k\ge k_2$. Inserting this
  in $S_k(1+\varepsilon_k)=0$ yields $\varepsilon_k=\bigOh(1/a^k)$. Using
  $S_k(1+\varepsilon_k)=0$ again shows that
  \begin{equation*}
    \varepsilon_k=\frac{\delta}{a^k}\bigl(1+\bigOh(k/c^k)\bigr),
  \end{equation*}
  where $\min\{a, \lvert a_2\rvert/a\}= a^\rho=: c>1$. 

Since
$G_{k}$ and $G$ are rational functions, $G_{k}$ and $G$ can be continued
analytically beyond their dominant singularities $\zeta_{k}$ and $1$, respectively. We have
$\lim_{k\rightarrow\infty}\Res_{z=\zeta_{k}}G_{k}(z)=\Res_{z=1}G(z)=1$. Thus
\cite[Theorem~1]{Prodinger-Wagner:ta:boots} implies \eqref{eq:2} and the
limiting distribution.

The coefficients $w_{\ell k}$ and $w_{\ell}$ of $z^{\ell}$ in $G_{k}$ and $G$,
respectively, coincide for $k\geq \ell$. Thus the support of $X_{\ell}$ is
finite. Furthermore, the condition on $s_{0}$ implies that $w_{\ell}=1+\bigOh(\kappa^{\ell})$ for a constant
$0\leq\kappa<1$ by singularity analysis. Thus, the expectation is 
\begin{equation}\label{eq:7}
  \mathbb E X_{\ell}=\sum_{k\geq 0}k\mathbb P(X_{\ell}=k)=\sum_{k=
    0}^{\ell}\Big(1-\frac{w_{\ell k}}{w_{\ell}}\Big)=\sum_{k=0}^{\ell}(1-w_{\ell
    k})+\bigOh(\ell \kappa^{\ell}).
\end{equation}

  Using the residue theorem and the assumption $r_0(1)=s_0(1)$, we get 
  \begin{align*}
    w_{\ell k}&=\Res_{z=0}\frac{R_k(z)}{z^{\ell+1}S_k(z)}\\
    &= \frac1{2\pi i}\oint_{\lvert z\rvert=1+C/2} \frac{R_k(z)}{z^{\ell+1}S_k(z)} - 
    \Res_{z=\zeta_k}\frac{R_k(z)}{z^{\ell+1}S_k(z)}\\
    &= -\frac{R_k(\zeta_k)}{S'_k(\zeta_k)}\zeta_k^{-(\ell+1)} + \bigOh((1+C/2)^{-\ell})\\
    &= \exp(-\ell\delta/a^k)\bigl(1+\bigOh(k/a^k)+\bigOh(\ell k/(a^kc^k))\bigr)+ \bigOh((1+C/2)^{-\ell})\end{align*}
  for $k_3\le k\le n$.

Replacing $w_{\ell k}$ with $\exp(-\ell\delta/a^{k})$ yields the error terms
\begin{equation}\label{eq:8}
  \lvert w_{\ell k}-\exp(-\ell\delta/a^{k})\rvert=
  \begin{cases}
    \bigOh(\ell^{-2})&\text{for }0\leq k\leq\log_{a}(\ell\delta/(4 \log \ell)),\\
    \bigOh(\log_{a}^{\rho+2}\ell/\ell^{\rho})&\text{for }\log_{a}( \ell\delta/(4\log
    \ell))\leq k\leq 5\log_{a}\ell,\\
    \bigOh(\ell^{-3})&\text{for }5\log_{a}\ell\leq k\leq \ell
  \end{cases}
\end{equation}
where we used $w_{\ell k} \le w_{\ell k_3}$ for $k\le k_3$.
As
$1-\exp(-\ell\delta/a^{k})$ is exponentially small for $k>\ell$, we obtain
\begin{equation}\label{eq:sum-main-expectation}
  \sum_{k=0}^\ell (1-w_{\ell k})=\sum_{k=0}^\infty \bigl(1-\exp(-\ell \delta/
    a^k)\bigr) + O\left(\frac{\log^{\rho+3} \ell}{\ell^\rho}\right).
\end{equation}
Thus, \eqref{eq:3} follows  from \eqref{eq:7}, \eqref{eq:sum-main-expectation} and the well known fact (see e.g.\ \cite{Flajolet-Gourdon-Dumas:1995:mellin}) that
  \begin{align*}
    \sum_{k\ge0}\left(1-e^{-
        x/a^k}\right)
    =\log_ax+\frac{\gamma}{\log a}+\frac12+\Psi_{0}(\log_ax)+
    \bigOh(x^{-1})
  \end{align*}
  with the periodic function $\Psi_{0}(x)$ given in \eqref{eq:le:asymptotik:psi}.

The second moment is
\begin{equation}\label{eq:4}
  \begin{aligned}
    \mathbb E X_{\ell}^{2}&=\sum_{k\geq 0}k^{2}\mathbb
    P(X_{\ell}=k)=\sum_{k=0}^{\ell}(2k+1)\Big(1-\frac{w_{\ell
        k}}{w_{\ell}}\Big)\\
    &=\sum_{k=0}^{\ell}(2k+1)(1-w_{\ell
      k})+\bigOh(\ell^{2}\kappa^{\ell}).
  \end{aligned}
\end{equation}

As $\sum_{k=0}^{\ell}(1-w_{\ell k})$ has already been computed for the
expectation, we are left with $\sum_{k=0}^{\ell}k(1-w_{\ell k})$.
We use \eqref{eq:8} to obtain
\begin{equation}\label{eq:5}
  \sum_{k=0}^{\ell}k(1-w_{\ell k})=\sum_{k\geq 0}k(1-\exp(-\ell\delta/a^{k}))+\bigOh\Big(\frac{\log^{\rho+4}\ell}{\ell^{\rho}}\Big).
\end{equation}

The Mellin transform (see \cite{Flajolet-Gourdon-Dumas:1995:mellin}) of the
harmonic sum
$F(x)=\sum_{k\geq 0}k(1-\exp(-x/a^{k}))$ is 
\begin{equation*}
F^{*}(s)=\frac{-a^{s}}{(1-a^{s})^{2}}\Gamma(s)
\end{equation*}
for $-1<\Re s<0$. The singular expansion of this Mellin transform at $\Re s=0$
is
\begin{align*}
  F^{*}(s)&\asymp
  -\frac{1}{\log^{2}a}s^{-3}+\frac{\gamma}{\log^{2}a}s^{-2}+\Big(\frac
  1{12}-\frac{1}{2\log^{2}a}\Big(\gamma^{2}+\frac{\pi^{2}}{6}\Big)\Big)s^{-1}\\
  &\quad- \sum_{n\neq
  0}\frac{\Gamma(-\chi_{n})}{\log^{2}a} (s+\chi_{n})^{-2}-\sum_{n\neq 0}\frac{\Gamma'(-\chi_{n})}{\log^{2}a} (s+\chi_{n})^{-1}
\end{align*}
for $\chi_{n}=\frac{2\pi i n}{\log a}$. Thus,
\begin{align*}
  F(x)&=\frac 12\log_{a}^{2}x+\frac{\gamma}{\log
    a}\log_{a}x-\frac{1}{12}+\frac{1}{2\log^{2}a}\Big(\gamma^{2}+\frac{\pi^{2}}{6}\Big)\\
  &\quad-\frac{\log_{a}x}{\log
  a}\sum_{n\neq0}\Gamma(-\chi_{n})\exp(2\pi in\log_{a}
x)\\
&\quad+\frac{1}{\log^{2}a}\sum_{n\neq0}\Gamma'(-\chi_{n})\exp(2\pi in\log_{a}
x)+\bigOh(x^{-1}).
\end{align*}

Thus, $\mathbb VX_{\ell}=\mathbb EX_{\ell}^{2}-(\mathbb EX_{\ell})^{2}$,
\eqref{eq:4}, \eqref{eq:7}, \eqref{eq:3} and \eqref{eq:5} give the variance as stated in \eqref{eq:6}.
\end{proof}

\begin{theorem}\label{thm:ssde-neumann}
  Let $q\geq 2$ be even. Then the expected number of iterations when adding
  two SSDE of length $\ell$ with von Neumann's algorithm is
  \begin{equation}\label{eq:1}
    \log_{q}\ell+\log_{q}\delta+\frac{\gamma}{\log q}+\frac 52+\Psi_{0}(\log_{q}\ell+\log_{q}\delta)+\bigOh(\ell^{-1}\log^{4}\ell)
  \end{equation}
where
  \begin{equation*}
    \delta=\frac{(q - 1)(4q^{10} + 10q^9 + 18q^8 - 4q^7 - 10q^6 + 7q^5 + 44q^4 -
      29q^3 - 8q^2 - 20q + 16)}{4q^{3}(q + 1)^2(4q^7 - q^5 - 6q^4 + 8q^3 + 2q
      - 4)},
  \end{equation*}
  $\Psi_{0}(x)$ is a $1$-periodic function with mean $0$ given by the
  Fourier expansion
  \begin{equation}\label{eq:psi}
    \Psi_{0}(x)=-\frac{1}{\log q}\sum_{k\neq 0}\Gamma\Big(-\frac{2k\pi i}{\log
      q}\Big)e^{2k\pi ix}.
  \end{equation}
  The variance of the number of iterations is
  \begin{equation}\label{eq:var}
    \frac{\pi^{2}}{6\log^{2}q}+\frac
    {1}{12}+\Psi_{1}(\log_{q}\ell+\log_{q}\delta)-\frac{2\gamma}{\log q}\Psi_{0}(\log_{q}\ell+\log_{q}\delta)-\Psi_{0}^{2}(\log_{q}\ell+\log_{q}\delta)+\bigOh(\ell^{-1}\log^{5}\ell)
  \end{equation}
  where $\Psi_{1}$ is a $1$-periodic function with mean $0$ given by the Fourier expansion
  \begin{equation}
    \label{eq:psi-1}
    \Psi_{1}(x)=\frac{2}{\log^{2}q }\sum_{k\neq 0}\Gamma'\Big(-\frac{2k\pi i}{\log
    q}\Big)e^{2k\pi ix}.
  \end{equation}
The asymptotic formula
\begin{equation*}
  \mathbb P_{\ell}(t(\bfx,\bfy)\leq k)=\exp(-\delta\ell/q^{k})(1+o(1))
\end{equation*}
holds as $\ell\rightarrow\infty$ for $k=\log_{q}\ell+\bigOh(1)$.
The random variable $t(\bfX,\bfY)-\log_{q}\ell$ converges weakly to a
double exponential random variable if $\ell$ runs through a subset of the
positive integers such that the fractional part $\{\log_{q}\ell\}$ converges.
\end{theorem}

\begin{remark}
  A similar result for $q\geq
  4$ was
  obtained in \cite{Heuberger-Prodinger:2003:carry-propag} using the same
  probabilistic model as in Remark~\ref{rem:ignore-dep}. This changes the
  main term of the expected value. In \cite{Heuberger-Prodinger:2003:carry-propag}, the
  logarithm of the main term  was taken to the base $\alpha^{-1}$ with
  $\alpha=q^{-1}-q^{-4}+\bigOh(q^{-5})$. In contrast, we here obtain the
  logarithm of the main term in~\eqref{eq:1} to the base $q$, which is a more
  natural constant appearing in this context.

  For $q=2$, this result is contained in \cite{Heuberger-Prodinger:2003:carry-propag}.
\end{remark}

\begin{proof}
Let $\mathbb P_{\ell}$ be the (exact) equidistribution of all SSDE of length
$\ell$. For $k>\ell+2$, we know that $\mathbb P_{\ell}(t(\bfX, \bfY)\leq k)=1$
because an input sequence of length $\ell$ traverses at most $\ell$ solid edges
in the automaton in Figure~\ref{aut:ssde-von-neu}. 

If we use the approximate
equidistribution $W_{\ell}=(1+\bigOh(\xi^{\ell}))\mathbb P_{\ell}$ of all SSDE of
length $\ell$, an exponentially small error term is introduced. Because of the
finite support, this error term does
 not change the main term of the expectation, the variance and the
 distribution function. Thus, also the limiting distribution remains the same.

 We will use Theorem~\ref{theorem:bootstrapping} with the generating function
  \begin{equation*}
G_{k}(z)=\sum_{\ell\geq 0}w_{\ell k}z^{\ell}
\end{equation*}
for $w_{\ell k}=W_{\ell}(t(\bfx,\bfy)-2\leq k)$, $k\geq 0$.
To construct this generating function, we use the same techniques as in
\cite{Heuberger-Prodinger:2003:carry-propag}.

The generating function $G_{k}(z)$ counts the weighted number of paths in
the automaton
$\calN_{\SSDE}$ of the pattern $\ldots\mathcal B^{+}\mathcal
R^{\{1, k\}}\mathcal B^{+}\mathcal R^{\{1, k\}}\ldots$ where $\mathcal B^{+}$ is an arbitrary non-empty
sequence of dotted transitions and $\mathcal R^{\{1, k\}}$ is a non-empty sequence
of solid transitions of length at most $k$. The first transition can be a dotted
or a solid transition. We stop with either arbitrarily many dotted transitions
or at most $k$ solid transitions, where we have to take into account the
special situation in  states $2$ or $7$ in the
automaton in Figure~\ref{aut:ssde-von-neu} (see also Remark~\ref{rem:read-inf-sequ}): Because of the solid transition
starting in $2$ and $7$
with label $0$, we
are not allowed to stop with $k$ solid transitions in state $2$ or $7$ but
only with at most $k-1$ ones.

To find the generating functions for $\mathcal B^{+}$ and $\mathcal R^{\{1, k\}}$,
we use the transition matrices for the dotted and the solid parts of the
automaton $\calN_{\SSDE}$.

  Let $q\geq 6$.
  The transition matrix $R$ for the solid transitions of automaton $\calN_{\SSDE}$
  is a $12\times 12$ matrix given  in Table~\ref{tab:ssde-neumann-solid} in the
  appendix (using Lemma~\ref{lem:recursion-N}).
The transition matrix $B$ for the dotted transitions of automaton $\calN_{\SSDE}$ is
given in Table~\ref{tab:ssde-neumann-dotted} (using
Lemma~\ref{lem:recursion-N}). The order of the states is given in
\eqref{eq:states-ord} and also in Table~\ref{tab:ssde-neumann-exit-weights} in
the appendix.

The (matrix) generating function for arbitrary non-empty dotted paths $\mathcal B^{+}$ is
\begin{equation*}
  B^{+}(z)=(I-zB)^{-1}-I.
\end{equation*}
The entry $(i,j)$ of this matrix is the generating function of non-empty
dotted paths of arbitrary length starting in state $i$ and leading to state $j$.
For arbitrary non-empty solid paths, the (matrix) generating function is
\begin{equation*}
  R^{+}(z)=(I-zR)^{-1}-I.
\end{equation*}

To obtain the (matrix) generating function $R^{\{1, k\}}$ for non-empty solid paths $\mathcal R^{\{1,
  k\}}$ of length at most $k$, we have to restrict each entry of $R^{+}$
corresponding to an infinite geometric series to a finite geometric
series.\footnote{It is also possible to use $R^{\{1,k\}}(z)=zR+\cdots+z^{k}R^{k}=(I-z^{k+1}R^{k+1})(I-zR)^{-1}-I$. However, this involves a power
  of the \emph{symbolic} matrix $R$ with the \emph{symbolic} exponent $k$. This would require a full symbolic
  eigenvalue decomposition of $R$. The approach chosen here avoids this by
  introducing the length restriction on each entry individually.} We
will illustrate this procedure on
\begin{equation}\label{eq:entry-rcal}
  \frac{- q^{4} z + 10 q^{3} z - 3 q^{2} z^{2} - 24 q^{2} z + 10 q z^{2} +
8 z^{2}}{-8 q^{4} + 8 q^{3} z - 8 q^{2} z + 8 q z^{2}},
\end{equation}
the entry at position $(5,1)$ of $R^{+}$. The partial fraction
decomposition of~\eqref{eq:entry-rcal} with respect to $z$ is
\begin{equation*}
  -\frac{(3q+2)(q-4)}{8  q}+ \frac{(q-4)(q+4)
}{4  (q + 1)}\cdot\frac{1}{ 1 + z/q^{2}} +\frac{(q-1)(q-2)(q-4)}{8  q (q +
1)}\cdot\frac{1}{1 - z/q}.
\end{equation*}
By truncating the infinite geometric sum $(1-z)^{-1}$ after $k+1$ summands,
i.e.,  by replacing it with
$(1-z^{k+1})(1-z)^{-1}$, we obtain
\begin{equation*}
    -\frac{(3q+2)(q-4)}{8  q}+ \frac{(q-4)(q+4)}{4  (q + 1)}\cdot\frac{1-(-z/q^{2})^{k+1}}{ 1 + z/q^{2}} +\frac{(q-1)(q-2)(q-4)}{8  q (q +
1)}\cdot\frac{1-(z/q)^{k+1}}{1 - z/q}.
\end{equation*}

Let 
\begin{equation*}
  M_{k}(z)=\begin{pmatrix}
    0&B^{+}(z)\\R^{\{1, k\}}(z)&0
  \end{pmatrix}
\end{equation*}
be the block matrix of total size $24\times 24$.
Then, the (matrix) generating function of non-empty paths $\ldots\mathcal B^{+}\mathcal R^{\{1,
  k\}}\mathcal B^{+}\mathcal R^{\{1, k\}}\ldots$ is
\begin{equation*}
  (I-M_{k}(z))^{-1}-I.
\end{equation*}

To take into account the initial states and the exit weights in the automaton
$\calN_{\SSDE}$, we define the initial vector
\begin{equation*}
  u=(0, 0, 0, 0, 0, 0, 0, 0, 1, 0, 0, 0;0, 0, 0, 0, 0, 0, 0, 0, 1, 0, 0, 0).
\end{equation*}
Using the exit weights in Table~\ref{tab:ssde-neumann-exit-weights} in the
appendix, we further define the exit vector
\begin{multline*}
  v^{\top}=\Big(\frac{q+1}{q+2}\Big)^{2} (4, 1, 2, 0, 1, 0, 2, 2, 1, 1, 4, 2;
  4, 1, 2, 2, 1, 1, 2, 2, 1, 1, 4, 2)^{\top}\\+\Big(\frac{q+1}{q+2}\Big)^{2}M_{k-1}(z)( 0,0,0,2,0,1,0,0,0,0,0,0;    0,0,0,0,0,0,0,0,0,0,0,0)^{\top}
\end{multline*}
taking into account the special situation with states $2$ and $7$ in the automaton in
Figure~\ref{aut:ssde-von-neu}.

Then, the generating function is
\begin{align*}
  G_{k}(z)=u((I-M_{k}(z))^{-1}-I)v+\Big(\frac
  {q+1}{q+2}\Big)^{2}
\end{align*}
where we add the exit weight of state $(1,(0,0))$ because the empty word  was not counted until now.
The result is
\begin{align}
  \label{eq:gf}
  G_{k}(z)=\frac{r_{0}(z)+\big(\frac{z}{q}\big)^{k}r_{1}\big(z,
    \big(\frac{z}{q}\big)^{k}, \big({-}\frac
    z{q^{2}}\big)^{k}\big)}{(1-z)s_{0}(z)+\big(\frac zq\big)^{k}s_{1}(z)+\big({-}\frac z{q^{2}}\big)^{k}s_{2}\big(z,
    \big(\frac z q\big)^{k}, \big({-}\frac z{q^{2}}\big)^{k}\big)}
\end{align}
with
\begin{align*}
r_{0}(z)&=4 q^{7}  (q + 1)^{3}   (4 z^{2} - 3 q^{2} z- q^{3})
  (  2 q z^{4}- 4 z^{4}+ 8 q^{3} z^{2}- q^{5} z^{2} - 6 q^{4} z^{2}+4 q^{7} ),\\
s_{0}(z)&=4 q^{7}  (q + 1)  (q + 2)^{2}  (z+q) 
 (z- q^{2})  (2 q z^{4}- 4 z^{4}+8q^{3} z^{2} -  q^{5} z^{2} - 6 q^{4} z^{2}
 +4 q^{7} ),\\
s_{1}(z)&=- (q + z)  z^2 (q + 2)^2  q^4  (4q^{12} + 6q^{11}z + 2q^{10}z^2 -
4q^{11} - 24q^{10}z\\
&\quad - 8q^9z^2 + 24q^{10} + 26q^9z + 4q^8z^2 + 5q^7z^3 - 7q^6z^4 - 48q^9 \\
&\quad- 20q^8z + 18q^7z^2 - 9q^6z^3 + 34q^5z^4 + 5q^4z^5 + 32q^8 + 36q^6z^2 \\
&\quad- 32q^5z^3 - 59q^4z^4 - 25q^3z^5 - 112q^5z^2 + 84q^4z^3 + 40q^3z^4 \\
&\quad+ 44q^2z^5 + 64q^4z^2 - 48q^3z^3 + 4q^2z^4 - 36qz^5 - 16qz^4 + 16z^5)
\end{align*}
  and some polynomials $r_{1}$ and $s_{2}$ in $z$, $(z/q)^{k}$ and $(-z/q^{2})^{k}$
with coefficients in $\mathbb Q[q]$.
The polynomial $s_{0}$ does not have any zeros in the closed unit disc. We
have $r_{0}(1)\neq 0$ and $\delta >0$.

For $q\leq 4$, the construction of the generating function $G_{k}(z)$ is the same, only the
matrices $R$ and $B$ and the vectors $u$ and $v$ are slightly
different. Nevertheless, \eqref{eq:gf} including the definitions of all the occurring polynomials is still valid.

By Theorem~\ref{theorem:bootstrapping}, we obtain the expectation, the
variance, the distribution function and the limiting distribution of the
non-negative truncation of
$t(\bfx,\bfy)-2$. From \eqref{eq:8} and the monotonicity of $w_{\ell k}$, we
know that
$w_{\ell-2}=w_{\ell-1}=\bigOh(\ell^{-2})$. Therefore, the results
transfer to the random variable $t(\bfx, \bfy)$ as stated in the theorem.
\end{proof}

\bibliographystyle{amsplainurl}
\bibliography{cheub}

\newpage
\begin{appendix}
\section{Transition Matrices}
\label{app:transition-matrices}
  
\begin{table}[h]
  \centering
  \begin{equation*}\setlength{\extrarowheight}{4pt}
  \begin{array}{ccc}
  (-1,(0,0))    &(0,(0,0))                      &(1,(0,0))\\\hline
  (d - 1)(d - 2)y& -2d^2 - 2dq + q^2 + 6d + 3q -4 & (d + q - 1)(d + q - 2)x\\
  (d - 1)dy       &    -2d^2 - 2dq + q^2 + 2d +q  & (d + q)(d + q - 1)x\\
  (d + 1)d y      &   -2d^2 - 2dq + q^2 - 2d -q   & (d + q + 1)(d + q)x
\end{array}
\end{equation*}

\caption{Transition matrix of $\calS_{(q,d)}$ in
  Section~\ref{sec:standard-addition-q-d} multiplied with $2q^{2}$. The order
  of the states is given in the first line.}
\label{tab:q-d-stand}
\end{table}

\begin{landscape}
\begin{table}
  \centering\scriptsize
  \begin{equation*}\setlength{\extrarowheight}{2pt}
  \begin{array}{*{14}{c}}
6q^2 - 12q + 8&4&4(q - 2)y&8&4q - 8&4q - 8&8&4(q - 2)x&(q - 2)(q - 4)y&(q - 2)(q - 4)x&2y&2x&4q - 8&4q - 8\\
8q^2&0&0&0&0&0&0&0&0&0&0&0&0&0\\
6(q - 2)q&0&4qy&0&0&4q&0&0&2(q - 2)qy&0&0&0&8q&0\\
2(qy + 2q - 2y)q&0&4qy^2&0&0&4qy&0&0&2(q - 2)qy^2&0&0&0&0&0\\
2(3q - 2)q&0&4(q - 2)y&8&0&4q - 8&8&0&2(q - 2)(q - 4)y&0&0&0&8q - 16&0\\
2(3q - 2)q&0&0&8&4q - 8&0&8&4(q - 2)x&0&2(q - 2)(q - 4)x&0&0&0&8q - 16\\
2(qx + 2q - 2x)q&0&0&0&4qx&0&0&4qx^2&0&2(q - 2)qx^2&0&0&0&0\\
6(q - 2)q&0&0&0&4q&0&0&4qx&0&2(q - 2)qx&0&0&0&8q\\
6(q - 2)q&4&4qy&8&4q - 16&4q&8&4(q - 4)x&(q - 2)qy&(q - 4)(q - 6)x&2y&2x&4q&4q - 16\\
6(q - 2)q&4&4(q - 4)y&8&4q&4q - 16&8&4qx&(q - 4)(q - 6)y&(q - 2)qx&2y&2x&4q - 16&4q\\
4(q - 2)q&0&0&0&0&0&0&0&4(q - 2)qy&0&0&0&16q&0\\
4(q - 2)q&0&0&0&0&0&0&0&0&4(q - 2)qx&0&0&0&16q\\
(3qy + 3q - 6y - 2)q&4&4qy^2&0&4q - 8&4qy&0&4(q - 2)x&(q - 2)qy^2&(q - 2)(q - 4)x&2y&2x&0&0\\
(3qx + 3q - 6x - 2)q&4&4(q - 2)y&0&4qx&4q - 8&0&4qx^2&(q - 2)(q -
4)y&(q - 2)qx^2&2y&2x&0&0
  \end{array}
\end{equation*}

  \caption{Transition matrix of $\calS_{\SSDE}$ for $q\geq 8$ in
    Section~\ref{sec:stand-addit-ssde} multiplied with $8q^{2}$. The order
  of the states is $\{(0,(0, 0))\}$,
$\{(0,(-1, 1)), (0,(1, -1))\}$,
$\{(-1,(-1, 0)), (-1,(0, -1))\}$,
$\{(-q/2,(-1, 0)), (-q/2,(0, -1))\}$,
$\{(0,(-1, 0)), (0,(0, -1))\}$,
$\{(0,(0, 1)), (0,(1, 0))\}$,
$\{(q/2,(0, 1)), (q/2,(1, 0))\}$,
$\{(1,(0, 1)), (1,(1, 0))\}$,
$\{(-1,(0, 0))\}$,
$\{(1,(0, 0))\}$,
$\{(-1,(-1, -1))\}$,
$\{(1,(1, 1))\}$,
$\{(-q/2,(0, 0))\}$,
$\{(q/2,(0, 0))\}$.}
  \label{tab:ssde-stand}
\end{table}

\begin{table}[h]
  \centering\scriptsize
  \begin{equation*}\setlength{\extrarowheight}{4pt}
  \begin{array}{*{12}{c}}
    (1,(-1, 1))&(4,(0, 0))&(5,(0, 1))&(2,(0, 1))&(5,(0, 0))&(2,(0,0))&(1,(-1, 0))&(3,(0, 1))&(1,(0, 0))\}&(3,(0, 0))&(4,(1, 1))&(4,(0, 1))\\\hline
  4&1&2&2&1&1&2&2&1&1&4&2
\end{array}
\end{equation*}
\caption{Exit weights of $\calN_{\SSDE}$ in
  Section~\ref{sec:asympt-analys-von} multiplied with
  $\big(\frac{q+2}{q+1}\big)^{2}$. As discussed in \eqref{eq:states-ord}, the
  states of $\calN_{\SSDE}$ are
  equivalence classes of states. For brevity, we list one representative for
  each state of $\calN_{\SSDE}$ to give the order of the
  states.}
\label{tab:ssde-neumann-exit-weights}
\end{table}
  \begin{table}
    \centering
    \begin{equation*}\setlength{\extrarowheight}{4pt}
      \begin{array}{*{12}{c}}
        0 & 0 & 0 & 0 & 0 & 0 & 0 & 0 & 0 &
        0 & 0 & 0 \\
        0 & 0 & 0 & 0 & 0 & 0 & 0 & 0 & 0 &
        0 & 0 & 0 \\
        0 & 0 & 0 & 0 & 0 & 0 & 0 & 0 & 0 &
        0 & 0 & 0 \\
        0 & 0 & 0 & 0 & 0 & 0 & 0 & 0 & 4q(q-2) & 8 q & 0 & 0 \\
        0 & 0 & 0 & 0 & 0 & 0 & 0 & 0 & 0 &
        0 & 0 & 0 \\
        4 &  (q-4)(q-6) & 0 & 8 & 0 & 4 (q - 4) & 4
        (q - 4) & 8 &  3q(q-2)& 4 q & 4 & 4 (q - 4) \\
        0 & 0 & 0 & 0 & 0 & 0 & 0 & 0 & 0 &
        0 & 0 & 0 \\
        0 &  2(q-2)(q-4) & 0 & 8 & 0 & 8 (q - 2) &
        4 (q - 2) & 8 &  2q(q-2) & 0 & 0 & 4 (q - 2) \\
        0 & 0 & 0 & 0 & 0 & 0 & 0 & 0 & 0 &
        0 & 0 & 0 \\
        0 &  (q-2)(q-4) & 0 & 8 & 0 & 4 (q - 2) & 4 (q
        - 2) & 8 &  (3q - 4)(q - 2)& 4 (q - 2) & 0 & 4 (q - 2)
        \\
        0 & 0 & 0 & 0 & 0 & 0 & 0 & 0 & 0 &
        0 & 0 & 0 \\
        0 & 0 & 0 & 0 & 0 & 0 & 0 & 0 & 0 &
        0 & 0 & 0
      \end{array}
    \end{equation*}
  \caption{Transition matrix $R$ for the solid transitions in $\calN_{\SSDE}$
    for $q\geq 6$ in
    Section~\ref{sec:asympt-analys-von} multiplied with
    $8q^{2}$. The order of the states is the
    same as in Table~\ref{tab:ssde-neumann-exit-weights}.}
  \label{tab:ssde-neumann-solid}
\end{table}
\begin{table}
  \centering
  \begin{equation*}\setlength{\extrarowheight}{4pt}
  \begin{array}{*{12}{c}}
    0 & 0 & 0 & 0 & 0 & 0 & 0 & 0 & 8
q^{2} & 0 & 0 &0\\
4 & 2 (q - 4)^{2} & 0 & 16 & 0 & 8 (q - 3) & 8 (q - 3) & 16 & 2  (3  q -
2) (q - 2) & 8  (q - 1) & 4&8(q-3) \\
0 & 2 (q - 2) (q - 4) & 0 & 8
& 0 & 8  (q - 2) & 4 (q - 2) & 8 & 2 q(3 q - 2)  & 0 & 0 &4(q-2)\\
0 & 2 (q - 2) (q - 4) & 0 & 8
& 0 & 8  (q - 2) & 4  (q - 2) & 8 & 2q (q -
2)  & 0 & 0 &4(q-2)\\
4 & 2 (q - 2)(q - 4) & 0 & 8
& 0 & 4 (q - 2) & 8  (q - 2) & 8 & 6  q^{2} - 12
 q + 8 & 4 (q - 2) & 4&8(q-2) \\
0 & (q - 2) (q - 4) & 0 & 8 &
0 & 4 (q - 2) & 4 (q - 2) & 8 & (3  q -
4)(q - 2) & 4( q - 2) & 0&4(q-2) \\
0 & 2 (q - 2) (q - 4) & 16 & 0
& 8  (q - 2) & 0 & 4  (q - 2) & 0 & 2q (3
 q - 2)  & 0 & 0 &4(q-2)\\
0 & 0 & 0 & 0 & 0 & 0 & 0 & 0 & 4 
q^{2} & 0 & 0 &0\\
4 & 2 (q - 2)(q - 4) & 16 & 0
& 8  (q - 2) & 0 & 8  (q - 2) & 0 & 6  q^{2} -
12  q + 8 & 0 & 4&8(q-2) \\
4 & (q - 2) (q - 4) & 0 & 0 &
0 & 0 & 4 (q - 2) & 0 & q(3 q - 2) 
& 0 & 4& 4(q-2)\\
0 & 4 (q - 2) (q - 4) & 0 & 0
& 0 & 16 (q - 2) & 0 & 0 & 4 q(q -
2)  & 16 q & 0&0\\
0&2(q - 2)(q - 4)&0&8&0&8(q - 2)&4(q - 2)&8z&6(q - 2)q&8q&0&4(q - 2)
  \end{array}
\end{equation*}
  \caption{Transition matrix $B$ for the dotted transitions in $\calN_{\SSDE}$
    for $q\geq 6$ in
    Section~\ref{sec:asympt-analys-von} multiplied with $8q^{2}$. The order of the states is the
    same as in Table~\ref{tab:ssde-neumann-exit-weights}.}
\label{tab:ssde-neumann-dotted}
\end{table}

\end{landscape}

\end{appendix}

\end{document}